\documentclass[final]{siamart0516}
\usepackage{braket}

{
  \theoremstyle{plain}
  \theoremheaderfont{\normalfont\itshape}
  \theorembodyfont{\normalfont}
  \theoremseparator{.}
  \theoremsymbol{}
  \newtheorem{assumption}{Assumption}
}

\newsiamremark{remark}{Remark}
\newsiamremark{notation}{Notation}

\newcommand{\RR}{\mathbb{R}}

\usepackage{amsfonts}
\usepackage{graphicx}
\usepackage{epstopdf}
\usepackage{algorithmic}
\ifpdf
  \DeclareGraphicsExtensions{.eps,.pdf,.png,.jpg}
\else
  \DeclareGraphicsExtensions{.eps}
\fi

\newcommand{\TheTitle}{Analysis of the Extended Coupled-Cluster
  Method in Quantum Chemistry} 
\newcommand{\TheAuthors}{A. Laestadius, and S. Kvaal}

\headers{Analysis of the Extended CC Method in Quantum Chemistry}{\TheAuthors}   

\title{{\TheTitle}\thanks{Submitted to the editors \today.
\funding{
	This work has received funding from the ERC-STG-2014 under grant
	agreement No 639508, and from the Research Council of Norway (RCN)
under CoE Grant No.~179568/V30 (CTCC).
}}}

\author{
  Andre Laestadius\thanks{Centre for Theoretical and Computational Chemistry, Department of Chemistry, University of Oslo, P.O.~Box 1033 Blindern, N-0315 Oslo, Norway
    (\email{andre.laestadius@kjemi.uio.no}).}
  \and
  Simen Kvaal\thanks{Centre for Theoretical and Computational Chemistry, Department of Chemistry, University of Oslo, P.O.~Box 1033 Blindern, N-0315 Oslo, Norway}
}

\usepackage{amsopn}

\ifpdf
\hypersetup{
  pdftitle={\TheTitle},
  pdfauthor={\TheAuthors}
}
\fi

\begin{document}

\maketitle

\begin{abstract}
  The mathematical foundation of the so-called extended
  coupled-cluster method for the solution of the many-fermion
  Schr\"odinger equation is here developed. We prove an existence and
  uniqueness result, both in the full infinite-dimensional amplitude
  space as well as for discretized versions of it. The extended
  coupled-cluster method is formulated as a critical point of an
  energy function using a generalization of the Rayleigh--Ritz
  principle: the bivariational principle. This gives a quadratic bound
  for the energy error in the discretized case. The existence and
  uniqueness results are proved using a type of monotonicity property
  for the flipped gradient of the energy function. Comparisons to the
  analysis of the standard coupled-cluster method is made, and it is
  argued that the bivariational principle is a useful tool, both for
  studying coupled-cluster type methods, and for developing new
  computational schemes in general.
\end{abstract}

\begin{keywords}
  quantum chemistry; coupled-cluster method; extended coupled-cluster method; bivariational principle; uniqueness and existence; error estimates
\end{keywords}

\begin{AMS}
  65Z05, 81-08, 81V55
\end{AMS}

\section{Introduction}
The coupled-cluster (CC) method is today the \emph{de facto} standard
wavefunction-based method for electronic-structure calculations, and
has a complex and interesting history
\cite{Paldus2005,Kummel1991,Cizek1991,Bishop1991}. To cut a long story
short, it was invented by Coester and K\"ummel in the 1950s as a
method for dealing with the strong correlations inside an atomic
nucleus \cite{Coester1958,Coester1960}. From nuclear physics, the idea
migrated to the field of quantum chemistry in the 1960s due to the
seminal work of researchers such as Sinano\u{g}lu, \v{C}\'{i}\v{z}ek,
Paldus and Shavitt \cite{Sinanoglu1962,Cizek1966,Paldus1972}. An
interesting turn of events is that the method returned to nuclear
physics in the 1990s, when Dean and Hjorth--Jensen applied the now
mature methodology to nuclear structure calculations \cite{Dean2004a}.

The main feature of the CC method is the use of an exponential
parametrization of the wavefunction. This ensures proper scaling of
the computed energy with system size (number of particles), i.e., the
method is \emph{size extensive}. 
At the same time, the CC method is only \emph{polynomially scaling} with
respect to system size. These factors have led to the popularity of
the method.

However, the theory does not satisfy the (Rayleigh-Ritz) variational
principle, i.e., the computed CC energy is not guaranteed to be an
upper bound to the exact energy. This has traditionally been the
main criticism of CC calculations, as an error estimate is not readily
available. Furthermore, in the original
formulation it was not \emph{variational} in the sense that the
solution was not formulated as a stationary point of some
function(al).

Helgaker and J\o rgensen later formulated the CC method in terms of a
Lagrangian \cite{Helgaker1988,Helgaker1989}, viewing the solution of
the CC amplitude equations as a constrained optimization of the 
energy, the set of cluster amplitude equations becoming
constraints. This is today the standard formulation of the CC method.

Already in 1983, Arponen \cite{Arponen1983} derived the so-called extended CC
method (ECC) from a generalization of the Rayleigh--Ritz
variational principle, the \emph{bivariational principle}. This
principle formally relaxes the condition of the Hamiltonian being
symmetric, and thus introduces the left eigenvector as a variable as
well as the right eigenvector. Arponen noted that the standard CC
method can be viewed as an approximation to ECC, and continued to write down the
standard CC Lagrangian. In the bivariational interpretation, Helgaker
and J{\o}rgensen's Lagrange multipliers are actually wavefunction
parameters on equal footing with the cluster amplitudes. No distinction
is being made.

Both Helgaker and J{\o}rgensen's CC Lagrangian and Arponen's
bivariational formulation cast CC theory in a \emph{variational}
(stationary point) setting. However, only the bivariational point of
view allows, at least formally, systematic improvement by adding other
degrees of freedom than the cluster amplitudes to the ansatz. The
bivariational principle is therefore of potential great use when
developing novel wavefunction-based methods, see for example
Ref.~\cite{Kvaal2012}, where the single-particle functions are
introduced as (bi)variational parameters in a time-dependent
setting. However, while the bivariational principle is rigorous, it is
not known how to introduce \emph{approximations} by parameterizations
of the wavefunctions, such that one can obtain existence and
uniqueness results as well as error estimates.

In this article, we will provide a rigorous analysis of a version of
the ECC method. The idea is, starting from the bivariational quotient,
to choose a function $\mathcal{F}$ (see Eq.~\eqref{Feq}) that is
(locally and strongly) monotone and where $\mathcal{F}=0$ is
equivalent to a critical point of the bivariational quotient. Until
now, ECC has not been turned into a practical tool in chemistry due to
its complexity. On the other hand, the analysis herein is a step
towards obtaining a rigorous foundation for the application of the
bivariational principle. We believe that the approach taken, by
showing the monotonicity of the flipped gradient $\mathcal{F}$, is an
approach that may allow existence and uniqueness results in
much more general settings.

We build our analysis on articles by Rohwedder and
Schneider, who fairly recently put the standard CC method on sound
mathematical ground
\cite{Schneider2009,Rohwedder2013,Rohwedder2013b}. They proved, among
other important results, a uniqueness and existence result of the
solution of the CC amplitude equations. The result rests on a certain
monotonicity property of the CC equations. Moreover, in
Ref.~\cite{Rohwedder2013} the boundedness of cluster operators (as
operators on a Hilbert space that guarantees finite kinetic energy)
was established, which turns out to be a rather subtle matter. They
also provided error estimates for the energy using the stationarity
condition of the Lagrangian.

This article is structured as follows: In Section \ref{sec:exp-ansatz}
we discuss the solution of the Schr\"odinger equation by employing an
exponential ansatz. We here present relevant results needed for this
work. In particular Lemma \ref{lemma:ECCpara} is the motivation for
our choice of ECC variables and links the ECC energy function to the
bivariational principle. Theorem \ref{thm:contECC} formulates the
continuous ECC equations and equates the solution of these equations
with the solution of the Schr\"odinger equation.

In Section \ref{class} we analyze the flipped gradient of the
ECC energy function and prove strong and local monotonicity for this
entity. This is achieved for two complementary situations. Theorem
\ref{thm:taylormono} proves this property under assumptions on the
structure of the solution, whereas Theorem \ref{thm:splittmono} under
assumptions on the Hamiltonian. Along the lines of the analysis of
Rohwedder and Schneider for the CC theory, we prove existence and
uniqueness for the solution of the (continuous) ECC equation and
truncated (discrete) versions of it, see Theorem \ref{UandE}.  This
theorem also guarantees convergence towards the full solution as the
truncated amplitude spaces tend to the continuous ones. Theorem
\ref{ThmEst} formulates a sufficient condition for the truncated
amplitude spaces to grant a unique solution of the discrete ECC
equation. Again the monotonicity is used for the flipped
gradient. Lastly, in Theorem \ref{thm:Eest} we obtain error
estimates for the truncated ECC energy.  The energy estimates are
obtained without the use of a Lagrangian and are instead based on the
bivariational formulation of the theory.

\section{Solving the Schr\"odinger equation using the exponential ansatz}
\label{sec:exp-ansatz}

\subsection{Traditional CC theory in a rigorous manner}
\label{sec:tcc}

In this section we consider the exponential parametrization for the
$N$-electron ground-state wavefunction $\psi_*$ satisfying the $N$-electron
Schr\"odinger equation (SE)
\begin{equation*}
H\psi_* = E_*\psi_*.
\end{equation*}
Here, $E_*$ is the
ground-state energy and $ H$ is the Hamiltonian of a molecule in the
Born--Oppenheimer approximation. We assume that $\psi_*$ exists and that
it is non-degenerate, and we denote by $\gamma_* > 0$ the spectral
gap (for definition see Section \ref{sec:aap}).

The set of admissible wavefunctions is a Hilbert space
$\mathcal{H} \subset\mathcal L_N^2$ of finite kinetic energy wavefunctions, with
norm $\|\psi\|_{\mathcal{H}}^2 = \|\psi\|^2 + \|\nabla\psi\|^2$. Here,
$\mathcal L_N^2$ is the space of totally antisymmetric square-integrable
functions $\psi : (\RR^{3}\times \{\uparrow,\downarrow\})^N \to
\RR$, with norm $\|\cdot\|$ and inner product $\braket{\cdot,\cdot\cdot}$. In this work, we restrict our attention to the \emph{real} space
$\mathcal L^2_N$, and thus real Hamiltonians.

We will furthermore assume that that the ground-state wavefunction $\psi_*$ is non-orthogonal to a (fixed)
reference determinantal wavefunction $\phi_0 \in \mathcal{H}$, and
thus, using intermediate normalization, we have
$\psi_* = \phi_0 + \psi_\perp$, where
$\braket{\phi_0,\psi_\perp} = 0$.

The molecular Hamiltonian has a set of useful properties that make
the SE well-posed~\cite{Yserentant2010}. The operator $ H :
\mathcal{H}\to\mathcal{H}'$ is a
bounded (continuous) operator into the dual $\mathcal{H}'$, i.e.,
there exists a constant $C \geq 0$ such that for all $\psi,\psi'\in
\mathcal{H}$,
\begin{subequations}\label{eqs:hamiltonian}
	\begin{equation}
	|\braket{\psi', H\psi}| \leq C \|\psi'\|_{\mathcal{H}}
	\|\psi\|_{\mathcal{H}}. \label{eq:bounded-operator}
	\end{equation}
	Moreover, $H$ is below bounded by a constant $e \in \RR$ such that $ H + e$ is
	$\mathcal{H}$-coercive, i.e., there exists a constant $c \geq 0$ such
	that for all $\psi \in \mathcal{H}$,
	\begin{equation}
	\braket{\psi,( H + e)\psi} \geq c \|\psi\|_{\mathcal{H}}^2.
	\label{eq:coercive-operator}
	\end{equation}
	The latter inequality is often referred to as a G{\aa}rding
	estimate and it is immediate that $e>-E_*$. Finally, $H$ is symmetric,
	\begin{equation}
	\braket{\psi,H\psi'} = \braket{\psi', H\psi}.
	\label{eq:symmetric-operator}
	\end{equation}
\end{subequations}
Equations~(\ref{eq:bounded-operator}--\ref{eq:symmetric-operator})
form assumptions on $ H$ that will be used frequently.

In a standard fashion, we introduce a basis for $\mathcal{H}$ of determinantal wavefunctions
built from the $N$ ``occupied'' functions $\chi_i$ (forming  $\phi_0$)
as well as ``virtual'' functions $\chi_a$,
$a=N+1,N+2,\cdots$. Assuming that $\{\chi_p : p=1,2,\cdots\}$ is an
$\mathcal L_1^2$-orthonormal basis, the corresponding determinantal basis
$\{\phi_\mu\}$ is $ \mathcal L_N^2$-orthonormal. Additionally, we must require $\|\nabla\chi_p\|<+\infty$.

Each $\phi_\mu$ can be written on the form $\phi_\mu = X_\mu \phi_0$,
where $ X_\mu$ is an operator that creates up to $N$ particle-hole
pairs, i.e., $ \{X_\mu \}_{\mu\neq 0}$ are excitation operators, and for an arbitrary $\psi \in\mathcal H$ with $\langle \phi_0,\psi\rangle =1$ we have
\begin{equation*}
\psi = \phi_0 + \sum_{\mu\neq 0} c_\mu \phi_\mu = ( I +
C)\phi_0,
\end{equation*}
with  $ C = \sum_{\mu \neq 0} c_\mu  X_\mu$ being a
\emph{cluster operator}. The sequence $c=\{c_\mu\}_{\mu\neq 0}$
consists of the 
corresponding \emph{cluster amplitudes}. 
One says that $\phi_0$ spans the ``reference space'' $\mathcal{P}:= \text{span} \{ \phi_0 \}$, while $\{\phi_\mu\}_{\mu\neq 0}$ forms a basis for
$\mathcal{Q} = \mathcal{P}^\bot$, the ``excluded space''. It is clear that $\mathcal P \oplus \mathcal Q = \mathcal H$. (Here $\mathcal P^\perp$ denotes the $\mathcal L_N^2$ orthogonal complement of $\mathcal P$, i.e., with respect to the inner product $\langle \cdot,\cdot\cdot \rangle$.)

We introduce the convention that to each cluster amplitude sequence 
$c=\{c_\mu\}_{\mu\neq 0}$,
$t=\{t_\mu\}_{\mu\neq 0}$, etc, the corresponding \emph{cluster operator}
is denoted by the capital letter, i.e., $ C = \sum_\mu c_\mu  X_\mu$,
$ T = \sum_\mu t_\mu  X_\mu$, etc. Cluster operators by definition
excludes $\mu = 0$, so unless otherwise specified, in the sequel, all sums over $\mu$ runs over excited
determinants only. Moreover, we group the excitations according to the number of
``particle-hole pairs'' they create, i.e., $ T
=  T_1 +  T_2 + \cdots +  T_N$, etc.

We follow Ref.~\cite{Rohwedder2013b} and introduce a Banach space of
cluster amplitudes (in fact it is a Hilbert space). We say that
$t \in \mathcal{V}$ if and only if
$\|t\|_{\mathcal{V}} := \| T\phi_0\|_{\mathcal{H}} < +\infty$. Thus,
$t \in \mathcal{V}$ if and only if $\{t_\mu\}$ are the amplitudes of a
wavefunction of finite kinetic energy in the excluded space, i.e.,
$T \phi_0 \in \mathcal Q$. We
remark that the space of cluster operators corresponding to amplitudes
from $\mathcal{V}$ only depends on the choice of the reference
$\phi_0$ (i.e. the space $\mathcal P$), and not on the choice of the
virtual orbitals $\{\chi_a \}$, as long as $\{\phi_\mu\}$
is an orthonormal basis of $\mathcal Q$.

If the Hilbert space was finite dimensional, every linear operator
would be bound\-ed,
and the exponential map $ T \mapsto e^{{T}}$ would always be well-defined.
A cornerstone of formal CC theory is therefore the
well-definedness of the exponential map for general Hilbert spaces and
cluster operators (see
Lemma 2.3 in \cite{Rohwedder2013b}):
\begin{theorem}[Rohwedder and Schneider, the exponential mapping]
	$ T$ and $ T^\dagger$ are bounded operators on $\mathcal H$ if and
	only if $t\in\mathcal V$. Moreover, the exponential map
	$ T \mapsto e^{ T}$ is a (Fr\'echet) $\mathcal{C}^\infty$ isomorphism between
	$\mathcal C := \{  T: t\in \mathcal V \}$ and
	$\mathcal C_0 := \{  I+  T: t\in\mathcal V \}$. For
	$\psi\in\mathcal H$ such that $\langle \phi_0,\psi\rangle=1$ there
	exists a unique $t\in\mathcal V$ such that $\psi = e^{ T} \phi_0$,
	depending smoothly on $\psi$. In particular the exponential map and its inverse are locally Lipschitz, i.e., for $s,t\in\mathcal V$ inside some ball, there exist constants $D,D'$ such that
	\begin{equation}
	\Vert s -t \Vert_{\mathcal V} \leq D \Vert e^S\phi_0 -e^T \phi_0 \Vert_{\mathcal H} \leq D'  \Vert s -t \Vert_{\mathcal V}.
	\label{ExpMap1}
	\end{equation}
	\label{Thm1RS}
\end{theorem}
\begin{remark}
	Note that the above theorem does not hold for a general subspace (truncation) $\mathcal V_d \subset \mathcal V$. To see this, let $\{\chi_p \}$ be an orthonormal set but not necessarily a (complete) basis and consider a subset $\mathcal V_d$ corresponding to only single excitations ($T =T_1$, $S=S_1$ etc.) and assume $N>1$. Then the relation $e^{T} = I + S$ implies $T_1 + T_1^2/2 + \dots +T_1^N/N! = S_1$. Thus, we can choose $T_1\neq 0$ such that $e^{T_1}\neq I + S_1$, for any single excitation $S_1$.   
	\label{RmkThm1RS}
\end{remark}
The CC ansatz uses that the exponential is a bijection
between the sets $\mathcal C$ and $\mathcal C_0$, such that $\psi_* =
e^{ T_*}\phi_0$ for some $ T_*$ satisfying $e^{ T_*}=  I+ C_*$. 
We then have (see Theorem 5.3 in \cite{Rohwedder2013}):
\begin{theorem}[Rohwedder and Schneider, continuous CC formulation]\label{thm:ccc}
	Under the assumptions on $ H$ stated in
	Eqs.~\eqref{eq:bounded-operator} and \eqref{eq:coercive-operator},
	$\psi_*= e^{ T_*}\phi_0$ solves $ H\psi_* = E_* \psi_*$ if and only if 
	\begin{equation}
	f(t_*) = 0, \quad \text{and} \quad E_\text{CC}(t_*) = E_*, \label{eq:cont-cc-eqs}
	\end{equation}
	where $f : \mathcal{V} \to \mathcal{V}'$ is given by
	\begin{align*}
	f_\mu(t) := \langle \phi_\mu, e^{- T}  H e^{ T}\phi_0 \rangle, 
	\end{align*}
	and where $E_{\text{CC}} : \mathcal{V} \to \RR$ is given by
	\begin{align*}
	E_\text{CC}(t) := \langle \phi_0, e^{- T}  H e^{ T}\phi_0 \rangle .
	\end{align*}
\end{theorem}
\begin{remark}
	(i) Equation~\eqref{eq:cont-cc-eqs} is the usual \emph{untruncated} amplitude and
	energy equations of CC theory, formulated in the infinite dimensional
	case, with $f: \mathcal{V} \to \mathcal{V}'$. 
	This formulation was derived and named the continuous CC 
	method in Ref.~\cite{Rohwedder2013}, being a mathematically rigorous formulation of
	the electronic SE using the exponential ansatz. Continuous here means that the excluded space $\mathcal Q$ is not discretized.
	
	(ii) A remark on a frequently used notation in this article is in
	place. Since $f(t)$ is an element of the dual space of $\mathcal V$, $f(t) \in \mathcal{V}'$, the pairing with
	any $s\in \mathcal{V}$ is continuous in $s$ and given by the infinite
	series $\braket{f(t),s} = \sum_\mu s_\mu
	f_\mu(t)$. It should be clear from context whether $\langle \cdot,\cdot\cdot\rangle$ refers to the $\mathcal L_N^2$ inner product or the just stated infinite series. 
\end{remark}

Even if Theorem~\ref{thm:ccc} reformulates the SE, it is not clear
that truncations of $ T$, either with respect to basis set or
excitation level (or both), will give discretizations that yield existence and
uniqueness of solutions as well as error estimates. The main tool here
is the concept of local strong monotonicity of
$f : \mathcal{V} \to \mathcal{V}'$. The following theorem is basically
a local application of a classical theorem by Zarantonello
\cite{Zarantonello1960}, see also Theorem~4.1 in
Ref.~\cite{Rohwedder2013b} and Theorem 25.B and Corollary 25.7 in
\cite{Zeidler1990}. We will have great use of this result when
studying the extended CC method of Arponen. Let $X$ be a Hilbert space
and define for a subspace $Y\subset X$ and $x\in X$ the distance
$d(Y,x)$ between $Y$ and $x$ by
\[
d(Y,x) := \inf_{y\in Y} \Vert y -x \Vert_X .
\]
We recall that if $Y$ is closed then there exists a minimizer $y_m$,
i.e., $d(Y,x)  = \Vert y_m -x \Vert_X $. This minimizer is the
orthogonal projection of $x$ onto $Y$. We now state without proof:
\begin{theorem}[Local version of Zarantonello's Theorem]
	\label{fCCthm}\label{thm:zarantonello}
	Let $f : X \to X'$ be a map between a Hilbert space $X$ and its dual
	$X'$, and let $x_*\in B_\delta$ be
	a root, $f(x_*)=0$, where $B_\delta$ is an open ball of
	radius $\delta$ around $x_*$.
	
	Assume that $f$ is Lipschitz continuous in $B_\delta$,  i.e.,
	that for all $x_1,x_2\in B_\delta$,
	\begin{equation*}
	\|f(x_1) - f(x_2)\|_{X'} \leq L \|x_1-x_2\|_{X},
	\end{equation*}
	for a constant $L$. Secondly, assume that $f$ is locally strongly
	monotone in $B_\delta$, i.e., that 
	\begin{equation*}
	\braket{f(x_1) - f(x_2), x_1-x_2} \geq \gamma \|x_1-x_2\|_X^2,
	\qquad \text{for all}\; x_1,x_2\in B_\delta,
	\end{equation*}
	for some constant $\gamma > 0$.

	Then, the following holds:
	\begin{enumerate}
		\item[1)]
		The root $x_*$ is unique in $B_\delta$. Indeed, there is
		a ball $C_\varepsilon\subset X'$ with $0\in C_\varepsilon$ such that the 
		solution map
		$f^{-1}: C_\varepsilon \to X$ exists and is Lipschitz continuous,
		implying that the equation
		\begin{equation*}
		f(x_* + \Delta x) = y
		\end{equation*}
		has a unique solution $\Delta x = f^{-1}(y) - x_*$,
		depending continuously on $y$, with norm $\|\Delta x\|_X \leq \delta$. 
		\item[2)]
		Moreover, let $X_d \subset X$ be a closed subspace 
		such that $x_*$ can be approximated sufficiently well, i.e., the distance
		$d(x_*,X_d)$ is small. Then, the projected problem
		$ f_d(x_d) = 0$ has a unique solution $x_d \in X_d \cap
		B_\delta$, and
		\begin{equation*}
		\|x_* - x_d\|_X \leq \frac{L}{\gamma} d(x_*,X_d).
		\end{equation*}
	\end{enumerate}
\end{theorem}

Rohwedder and Schneider proved under certain assumptions (see
Theorem 3.4 and 3.7 and Assumptions A and B in \cite{Rohwedder2013b}) that 
the amplitude equations $f : \mathcal{V}\to\mathcal{V}'$ are indeed
locally strongly monotone. (Lipschitz continuity follows from the
differentiability of $f$.) Thus, the second part of
Theorem~\ref{thm:zarantonello} then guarantees that the truncated CC
equations have a unique solution, and that the error tends to zero as
we increase the basis size and the truncation level of $ T$, if the
amplitude equation map $f$ is locally strongly monotone and Lipschitz
continuous.

Before addressing the extended CC method we follow
Helgaker and J{\o}rgensen \cite{Helgaker1988} and remark that one can view the CC method as minimization of $E_\text{CC}(t)$ over $\mathcal V$ under the constraint $f(t)=0$. The Lagrangian in this case becomes
\begin{equation}
\label{Lagrange}
\begin{split}
\mathcal L(t,s) &:= \langle \phi_0,e^{- T}  H e^{ T}\phi_0\rangle +
\sum_{\mu} s_\mu \langle  \phi_\mu,e^{- T}  H
e^{ T}\phi_0\rangle \\
& = \braket{\phi_0,( I +  S^\dag)e^{- T} H e^{ T}\phi_0},
\end{split}
\end{equation}
where $s = (s_\mu )_{\mu \neq 0} \in \mathcal{V}$ is the
multiplier, which can be gathered into an excitation operator
$ S = \sum_\mu s_\mu X_\mu$. Note that $D_{s_\mu} \mathcal L = f_\mu$
since
$\mathcal L(t,s) = E_\text{CC}(t) + \langle f(t),s\rangle$.  We shall in the next section see that the
Lagrangian formulation is contained in the bivariational formulation
of CC theory.

\subsection{The extended coupled-cluster method}

To link the forthcoming discussion to the previous section, we note that Arponen \cite{Arponen1983} derived the CC Lagrangian 
starting from the \emph{bivariational Rayleigh-Ritz quotient}
$\mathcal{E}_{\text{bivar}} : \mathcal{H}\times\mathcal{H}\to\mathbb{R}$,
\[
\mathcal{E}_{\text{bivar}}(\psi,\psi') :=  \frac{\langle \psi', H\psi \rangle}{\langle\psi',\psi\rangle}.
\]
Vis-\'a-vis the usual Rayleigh-Ritz quotient, $\psi$ and $\psi'$ are
here truly independent variables (not only treated as such in a formal
manner). (See also the discussion following Eq. (24) in \cite{Kvaal2013}.) 
The stationary condition $D\mathcal E_\text{bivar}=0$ yields the left
(and right) eigenvector(s) of $H$ with eigenvalue $E_*$, in fact, by
straight-forward differentiation we obtain the following result:
\begin{theorem}[Bivariational principle]\label{bivar}
	Let $ H : \mathcal{H}\to\mathcal{H}'$ be  a bounded operator. Then,
	$\mathcal{E}_\text{bivar}$ is an infinitely differentiable function
	at all points where $\braket{\psi',\psi}\neq 0$, and
	$D_{\psi}\mathcal{E}_\text{bivar} =
	D_{\psi'}\mathcal{E}_\text{bivar} = 0$ if and only if the left and
	right SE is satisfied,
	\[  H\psi = E\psi, \quad  H^\dag \psi' = E \psi', \quad \langle \psi',\psi \rangle
	\neq 0.\]
	Here, $H^\dag : \mathcal{H}\to \mathcal{H}'$ is defined by
	$\braket{H^\dag\psi',\psi} := \braket{\psi',H\psi}$.
\end{theorem}
\begin{remark} If we assume that $H$ satisfies all the
	requirements~(\ref{eq:bounded-operator}--\ref{eq:symmetric-operator}),
	in particular that $H$ is symmetric, the left and right eigenvalue
	problems become identical, being the weak formulation of the eigenvalue problem of a
	unique self-adjoint $\hat{H}$
	over $\mathcal{L}_N^2$. Suppose that $\hat{H}$ is close to self-adjoint,
	e.g., self-adjoint up to an $\mathcal{L}^2_N$-bounded perturbation. It is
	then reasonable that the left and right eigenvalue problems can be
	simultaneously solved (but with $\psi'\neq \psi$). Thus, the bivariational principle can be
	thought of as a generalization of Rayleigh--Ritz to at least certain
	non-symmetric problems.
\end{remark}

We now introduce an exponential ansatz also for the wavefunction
$\tilde{\psi}$. Following Arponen~\cite{Arponen1983}, we eliminate
the denominator by changing the normalization of $\psi'$, i.e., we set
$\tilde{\psi} = \psi'/\braket{\psi',\psi}$. The two scalar constraints lead to a smooth
submanifold $\mathcal{M}\subset\mathcal{H}\times\mathcal{H}$ of
codimension 2,
\begin{equation}\label{eq:M-def}
  \mathcal{M} := \left\{ (\psi,\tilde{\psi}) \in \mathcal{H}\times
    \mathcal{H} \; | \; \braket{\phi_0,\psi} =
    \braket{\tilde{\psi},\psi} = 1 \right\}.
\end{equation}
The next lemma shows that this manifold $\mathcal{M}$ can be
parameterized using cluster amplitudes.
\begin{lemma}[Extended CC parameterization]
	\label{lemma:ECCpara}
	Suppose $(\psi,\tilde{\psi})$ satisfies $\braket{\phi_0,\psi} =
	\braket{\tilde{\psi},\psi} = 1$. Then, there exists unique
	$(t,\lambda)\in \mathcal{V}\times\mathcal{V}$ depending smoothly on
	$(\psi,\tilde{\psi})\in\mathcal{M}$, such that
	\begin{equation*}
	\psi = e^{ T}\phi_0, \quad \text{and} \quad \tilde{\psi} =
	e^{- T^\dag} e^{ \Lambda} \phi_0,
	\end{equation*}
	which is a smooth map.
	In other words, the map $\Phi :
        \mathcal{V}\times\mathcal{V}\to\mathcal{M}$, $\Phi(t,\lambda) :=
	(\psi(t),\tilde{\psi}(t,\lambda))$ is a smooth map with a smooth
	inverse. 
\end{lemma}
\begin{proof}
	By Theorem~\ref{Thm1RS}, $t$ exists and is unique, depending
	smoothly on $\psi$ and vice versa.  Consider $\omega =
	e^{ T^\dag(\psi)} \tilde{\psi}$, which depends smoothly on
	$(\psi,\tilde{\psi})$. We have $\braket{\phi_0,\omega} = 1$, so by
	Theorem~\ref{Thm1RS} there exists a unique $\lambda$ depending
	smoothly on $\omega$, and hence $(\psi,\tilde{\psi})$, such that
	$\omega = e^{ \Lambda}\phi_0$. Now $\tilde{\psi} =
	e^{- T^\dag}e^{ \Lambda}\phi_0$, a smooth map of $(t,\lambda)$.
\end{proof}

We define the \emph{extended coupled-cluster energy functional}
$\mathcal{E} : \mathcal{V}\times\mathcal{V}\to \mathbb{R}$ by
$\mathcal{E} = \mathcal{E}_\text{bivar}\circ\Phi$, viz,
\begin{equation}
\mathcal{E}(t,\lambda) = \braket{\phi_0, e^{ \Lambda^\dag} e^{- T}  H
	e^{ T}\phi_0}. \label{eq:eccfun}
\end{equation}
Eq.~\eqref{eq:eccfun} defines Arponen's ECC energy
functional in a continuous, infinite dimensional formulation.
\begin{theorem}[Continuous extended coupled-cluster equations]
	\label{thm:contECC}
	Let the Hamiltonian $ H : \mathcal{H}\to \mathcal{H}'$ be as
	before. Then,
	\[  H\psi_* = E_*\psi_*, \quad\text{and}\quad  H\tilde{\psi}_* = E_*
	\tilde{\psi}_* \]
	with normalization $\braket{\phi_0,\psi_*} =
	\braket{\tilde{\psi}_*,\psi_*} = 1$, 
	if and only if $D\mathcal{E}(t_*,\lambda_*)=0$, i.e.,
	\begin{equation*}
	D_t \mathcal{E}(t_*,\lambda_*) = 0, \quad\text{and}\quad
	D_\lambda \mathcal{E}(t_*,\lambda_*) = 0,
	\end{equation*}
	where
	\begin{subequations}\label{Feq}
		\begin{align}
		D_{t_\mu}\mathcal{E}(t,\lambda) &=
		\braket{\phi_0,e^{ \Lambda^\dag}[e^{- T} H e^{ T}, X_\mu]\phi_0},\\
		D_{\lambda_\mu} \mathcal{E}(t,\lambda) &= \braket{\phi_\mu,
			e^{ \Lambda^\dag} e^{- T} H e^{ T}\phi_0},
		\end{align}
	\end{subequations}
	and where $(\psi_*,\tilde{\psi}_*) = \Phi(t_*,\lambda_*)$.
\end{theorem}
\begin{proof}
	$\Phi$ is differentiable with a
	differentiable inverse on $\mathcal{M}$, which is precisely the set of
	function pairs satisfying the normalization constraints. Thus
	$D\mathcal{E}(t_*,\lambda_*) = D[ \mathcal{E}_\text{bivar}\circ
	\Phi](t_*,\lambda_*)=0$ if and only if
	$D\mathcal{E}_\text{bivar}(\psi_*,\tilde{\psi}_*)=0$ with the side
	condition $\braket{\tilde{\psi}_*,\psi_*}=\braket{\phi_0,\psi_*}=1$.
	Moreover,
	$\mathcal{E}(t_*,\lambda_*) =
	\mathcal{E}_\text{bivar}(\psi_*,\tilde{\psi}_*) = E_*$.
	The formulas for the partial derivatives of $\mathcal{E}$ follow by
	elementary differentiation strategies.
\end{proof}

As in the case of standard CC theory, the continuous ECC equations do
not imply that truncations in amplitude or basis set gives a
well-behaved approximate method. To achieve this is the goal of the
next section.
\begin{remark}\label{canonical-remark}
	(i) We note that both $\psi$ and $\psi'$ are parameterized in an explicit
	multiplicatively separable manner, when the system is decomposed into
	non-interacting subsystems. This is the main advantage of the ECC
	parameterization. We observe that the CC Lagrangian (given by
	Eq.\eqref{Lagrange}) is obtained by a further change of variables
	$ S^\dag := e^{ \Lambda^\dag}-1$, which destroys this property of
	$\psi'$. Alternatively, one can view the CC Lagrangian as a
	first-order approximation to the ECC functional in terms of $\lambda$.

	(ii) Arponen defined a further change of variables through 
	$ t_\mu' = \langle \phi_0, e^{ \Lambda^\dagger}  X_\mu^\dagger   T \phi_0\rangle$, 
	and where the inverse $t = t( t',\lambda)$ is explicitly given by 
	$t_\mu = \langle \phi_0, e^{- \Lambda^\dagger} X_\mu^\dagger  { T}'
	\phi_0\rangle$, see Eqs. (5.6) and (5.7) in \cite{Arponen1983}. 
	The variables $( t', \lambda)$ turn out to be \emph{canonical} in the sense of
	classical Hamiltonian mechanics, i.e., the time-dependent
	Schr\"odinger equation is equivalent to Hamilton's equations of
	motion, 
	\begin{subequations}
		\begin{align}
		&i \dot{t}_\mu'=  D_{\lambda_\mu} \mathcal E', \nonumber\\
		& i\dot{\lambda}	_\mu =- D_{ t_\mu'} \mathcal E', \nonumber
		\end{align}
	\end{subequations}
	where $\mathcal{E}'(t',\lambda) :=\mathcal{E}(t(t',\lambda),\lambda)$
	and $\dot{t}$ (and $\dot{\lambda}$) denotes the time derivative of the amplitudes $t$ (and $\lambda$). 
	The canonical variables have a computational advantage over the earlier defined
	non-canonical variables. As it turns out, they introduce cancellations in
	the (linked) diagram series for $E_*$ compared to when using the
	non-canonical $(t,\lambda)$. 
	We shall not use the variables $(t',\lambda)$ here, as the
	analysis becomes considerably more complicated, and
	instead relegate their study to future work.
\end{remark}

\section{Analysis of ECC from monotonicity}\label{class}

\subsection{The flipped gradient $\mathcal{F}$}

We will discuss the stationary point of
$\mathcal E$ corresponding to the ground-state energy $E_*$ in terms of a map $\mathcal{F} : \mathcal{V}\times\mathcal{V}\to
\mathcal{V}'\times\mathcal{V}'$ defined by flipping the components of
the (Fr\'echet) derivative $D\mathcal{E} = (D_t\mathcal{E},D_\lambda\mathcal{E})$,
i.e., 
\begin{equation}
  \mathcal{F} :=  (D_t\mathcal{E},D_\lambda\mathcal{E}) \begin{pmatrix} 0 & 1 \\ 1 & 0 \end{pmatrix}
   = (D_\lambda\mathcal{E},D_t\mathcal{E}). \label{eq:flipping}
\end{equation}
The components
of the derivative are given in Eqs.~\eqref{Feq}.

For the forthcoming discussion, let $B_\delta(t,\lambda)$ denote the ball of radius $\delta>0$ centered at $(t,\lambda) \in \mathcal V\times \mathcal V$. Here the norm is $\Vert (\cdot,\cdot\cdot)\Vert_{\mathcal V \times \mathcal V}^2 := \Vert \cdot \Vert_{\mathcal V}^2 + \Vert \cdot\cdot\Vert_{\mathcal V}^2$. 
Let $(t_*,\lambda_*)\in\mathcal V\times \mathcal V$ be the optimal
amplitudes corresponding to the ground-state pair $(\psi_*,\tilde{\psi}_*)$, in particular $\mathcal F (t_*,\lambda_*)=0$.  
For the extended CC function $\mathcal F$ we now want to establish: 
\begin{enumerate}
	\item[(i)] $\mathcal F$ is locally Lipschitz, i.e., let $(t,\lambda)\in\mathcal V \times \mathcal V$ then there exists $\delta>0$ such that $(t_i,\lambda_i)\in B_{\delta}(t,\lambda)$ implies
	\begin{equation*}
	\label{Lip}
	\Vert \mathcal F(t_1,\lambda_1) -\mathcal F(t_2,\lambda_2) \Vert_{\mathcal V'\times \mathcal V'}\leq L \Vert (t_1,\lambda_1)-(t_2,\lambda_2) \Vert_{\mathcal V\times \mathcal V}
	\end{equation*}
	for some (Lipschitz) constant $L>0$, possibly depending
	only on $(t,\lambda)$ and $\delta$.
	\item[(ii)] $\mathcal F$ is locally and strongly monotone at $(t_*,\lambda_*)\in\mathcal V\times \mathcal V$, i.e., 
	there exists $\delta,\gamma >0$ such that
	\begin{equation*}
	\label{Mono}
	\langle \mathcal F(t_1,\lambda_1) -\mathcal F(t_2,\lambda_2), (t_1,\lambda_1) - (t_2,\lambda_2) \rangle
	\geq \gamma (\Vert t_1-t_2 \Vert_{\mathcal V}^2 + \Vert \lambda_1-\lambda_2 \Vert_{\mathcal V}^2)  
	\end{equation*}
	holds for all $(t_1,\lambda_1),(t_2,\lambda_2)\in B_\delta(t_*,\lambda_*)$.
\end{enumerate}
Item (i) above is readily established using the fact that $\mathcal
F$ is the flipped gradient of a smooth function. For (ii), we shall formulate two sets of assumptions (Assumption \ref{A1} and Assumption \ref{A2} below) that each is enough to give strong monotonicity for $\mathcal F$ locally at $(t_*,\lambda_*)$. 
Having proved (i) and (ii), we can apply
Theorem~\ref{thm:zarantonello} to obtain existence and uniqueness
results, also for truncated schemes.

The definition of local strong monotonicity of the map
$\mathcal{F}$ reduces to the existence of a $\gamma >0$ such that for $(t_i,\lambda_i)$ close
to $t_*,\lambda_*$, the quantity
\begin{equation}
\begin{split}
\Delta_1(t_1,\lambda_1,t_2,\lambda_2) +
\Delta_2(t_1,\lambda_1,t_2,\lambda_2) & := \braket{
	D_\lambda\mathcal{E}(t_1,\lambda_1) - D_\lambda \mathcal{E}(t_2,\lambda_2), t_1 -
	t_2} \\ &\quad + 
\braket{
	D_t\mathcal{E}(t_1,\lambda_1) - D_t\mathcal{E}(t_2,\lambda_2), \lambda_1 -
	\lambda_2} 
\end{split}
\label{Delta1212}
\end{equation}
satisfies
\begin{equation}
\Delta_2(t_1,\lambda_1,t_2,\lambda_2) + \Delta_2(t_1,\lambda_1,t_2,\lambda_2) \geq \gamma (\|t_1-t_2\|_\mathcal{V}^2 +
\|\lambda_1-\lambda_2\|_\mathcal{V}^2).\label{eq:delta-monotone}
\end{equation}

The choice of the map $\mathcal{F}$ can be motivated as follows: It is
clear that $D\mathcal{E}$ cannot be locally strongly monotone, as, just
like $\mathcal{E}_\text{bivar}$, all the critical points of
$\mathcal{E}$ are
intuitively saddle points (we will not prove this claim). On the other
hand, in Ref.~\cite{Rohwedder2013b}, the map $f(t)$ from
Theorem~\ref{thm:ccc} was considered, and demonstrated to be locally
strongly monotone under suitable assumptions. We observe that
$f = D_s \mathcal{L}$, a partial derivative of the Lagrangian, which
is \emph{linear} in $s$, so that $f$ is only a function of $t$. In Ref.~\cite{Rohwedder2013b} 
it was demonstrated that (locally at $t_*$)
\begin{equation}\label{eq:rscc-monotonicity}
\Delta(t_1, t_2) = \braket{[D_s \mathcal{L}](t_1) - [D_s
	\mathcal{L}](t_2), t_1-t_2} \geq \gamma
\|t_1-t_2\|_{\mathcal{V}}^2,
\end{equation}
for some constant $\gamma>0$. Thus, Eq.~\eqref{eq:rscc-monotonicity}
is ``half'' of the inequality~\eqref{eq:delta-monotone}. In extended CC,
the functional $\mathcal{E}$ is nonlinear in $\lambda$, indicating
that we should include $\lambda$ in the monotonicity
argument.

\subsection{Assumptions and preparation}
\label{sec:aap}
The analysis of Arponen's ECC method conducted here will be based on
two complementary assumptions, Assumption~\ref{A1} and
Assumption~\ref{A2}. The former deals with the accuracy of the
ansatz, i.e., the accuracy of the reference $\phi_0$, while the latter
considers a splitting of the Hamiltonian, e.g., the smallness of the
fluctuation potential when a Hartree--Fock reference is used. We thus
obtain two complementary monotonicity results applicable in different
situations. However, both assumptions rest on conditions on spectral gaps. Recall the $\mathcal P$ denotes the reference space and moreover set $\mathcal P_* := \text{span} \{\psi_*\}$. Let $ P$ and $ P_*$ denote the $\mathcal L_N^2$-orthogonal projections on $\mathcal P$ 
and $\mathcal P_*$, respectively. Essential for the analysis, we then either have to assume that: There exists $\gamma_*>0$ such that (Assumption~\ref{A1})
\begin{equation}
\langle (I-P_*)\psi, (H-E_*)(I-P_*)\psi \rangle \geq \gamma_* \Vert (I-P_*)\psi\Vert^2
\label{gammaStar}
\end{equation}
or there exists $\gamma_0>0$ such that (Assumption~\ref{A2})
\begin{equation}
\langle (I-P)\psi, (F-e_0)(I-P)\psi \rangle \geq \gamma_0 \Vert (I-P)\psi\Vert^2,
\label{gamma0}
\end{equation}
for all $\psi\in\mathcal H$. Here $F$ is a one-body operator that has $\phi_0$ as ground state with ground-state energy $e_0$. A Hamiltonian splitting is then given by $H= F + (H-F)$, and will be dealt with below in connection with Assumption~\ref{A2}.
We note that Eq.~\eqref{gammaStar} expresses the fact that $E_*$ is
the leftmost eigenvalue of $H$, that this eigenvalue exists, and has
multiplicity 1.

We iterate that throughout the analysis we assume that 
the system Hamiltonian is bounded as quadratic
form and additionally satisfying a G{\aa}rding estimate, see the
discussion in Section~\ref{sec:tcc}, and in particular
Eqs.~(\ref{eq:bounded-operator}--\ref{eq:symmetric-operator}). 
We first state a slight upgrade of Lemma 3.5 in \cite{Rohwedder2013b}. Note that for $\psi\in \mathcal H$, $(I - P)\psi \in\mathcal Q$. Also recall that in our notation $\Vert \cdot \Vert$ is the $\mathcal L_N^2$ norm.
\begin{lemma}
	\label{LemmaRoh}
	With $\psi_* = \phi_0 + \psi_\bot$, where $\psi_\bot \in \mathcal{Q}$ is the correction to $\phi_0$, we have:
	\begin{itemize}
		\item[(i)]Assume that \eqref{gammaStar} holds with $\gamma_*>0$ and that $\Vert \psi_\perp \Vert_{\mathcal H} < \varepsilon $. Then there
		exists a $\gamma_\varepsilon \in (0, \gamma_*]$ such that, for
		all $\psi \in \mathcal Q$
		\begin{equation}
		\label{EqRoh}
		\langle \psi, ( H-E_*) \psi \rangle \geq \frac{\gamma_\varepsilon}{\gamma_\varepsilon + e+E_*}c \Vert \psi\Vert_{\mathcal H}^2,
		\end{equation} 
		where $\gamma_\varepsilon \to \gamma_*$ as
                $\varepsilon \to 0+$. 
		\item[(ii)] Assume $F\phi_0 = e_0 \phi_0$ and that \eqref{gamma0} holds with $\gamma_0>0$ and that $F$ satisfies the G\aa rding estimate given in \eqref{eq:coercive-operator} (with constants $e_F$ and $c_F$). Then
		\begin{equation}
		\label{fock1}
		\langle \psi, ( F - e_0) \psi \rangle \geq \frac{\gamma_0}{\gamma_0 + e_F + e_0}c_F \Vert \psi \Vert_{\mathcal H}^2
		\end{equation} 
		for all $\psi\in \mathcal Q$. 
	\end{itemize}
\end{lemma}
\begin{proof}
	(i) Let $\psi \in \mathcal Q$. We first show that 
	for $\gamma_\varepsilon>0$ (and where $\gamma_\varepsilon\to \gamma_*$ as $\varepsilon\to 0+$) there holds
	\begin{equation}
	\label{L2part}
	\langle \psi, ( H-E_*) \psi \rangle \geq \gamma_\varepsilon\Vert \psi\Vert^2.
	\end{equation}
	Following the argument in the proof of Lemma 2.4 in \cite{Rohwedder2013b}, we then have with $0 < q:= \gamma_\varepsilon/(\gamma_\varepsilon + e + E_*)<1$ (recall that $e + E_*>0$ by necessity of the G\aa rding estimate)
	\begin{align*}
	\langle \psi, ( H-E_*) \psi \rangle&=q\langle \psi, ( H-E_*) \psi \rangle +(1-q)\langle \psi, ( H-E_*) \psi \rangle \\
	&\geq qc \Vert \psi\Vert_{\mathcal H}^2 + (\gamma_\varepsilon - q   (\gamma_\varepsilon + e + E_*   )   )\Vert \psi\Vert^2.
	\end{align*}
	Thus, if \eqref{L2part} holds we are done.
	
	Let $ P$ and $ P_*$ be as above. We use that
	\[
	\Vert  P -  P_*\Vert_{\mathcal B( \mathcal L_N^2)} \leq 2 \Vert \phi_0 - \psi_*'\Vert,
	\]
	where $\psi_*' = \psi_*/\Vert \psi_*\Vert$. Since $\psi_* = \phi_0 + \psi_\perp$, with 
	$\alpha := \Vert \psi_\perp \Vert$ we have
	\begin{align*}
	\Vert  P -  P_*\Vert_{\mathcal B(\mathcal L_N^2)} \leq 2\big( 2 - 2(1 + \alpha^2)^{-1/2} \big)^{1/2} =: j(\alpha).
	\end{align*}
	Note that $j(\alpha)$ is an increasing function for $\alpha>0$ and $j(\alpha) = 2\alpha + \mathcal O (\alpha^2)$. 
	
	Since $(H-E_*) P_* \psi =0$ (and $H$ is symmetric), the left-hand side of
	\eqref{L2part} equals 
	\[
	\langle ( I -  P_*)\psi, ( H-E_*) ( I- P_*)\psi \rangle,
	\] 
	which by \eqref{gammaStar} is bounded from below by $\gamma_*
	\Vert ( I- P_*)\psi \Vert^2$. Thus for $\alpha$ sufficiently small
	\begin{align*}
	\langle \psi, ( H-E_*) \psi \rangle &\geq  \gamma_* (\Vert ( I -  P)\psi\Vert - \Vert ( P- P_*)\psi \Vert )^2 \\
	& \geq \gamma_* ( 1 - j(\alpha))^2 \Vert \psi \Vert^2.
	\end{align*}
	Since $\varepsilon> \Vert \psi_\perp\Vert_{\mathcal H} \geq
	\alpha$, we have that \eqref{L2part} holds with
	$\gamma_\varepsilon := \gamma_*(1 -j(\varepsilon) )^2$. It is
	clear that $\gamma_\varepsilon\to \gamma_*$ as $\varepsilon$
	tends to zero from above because $j(\varepsilon) \to
	0$. 
	
	(ii) With $q_F:= \gamma_0/(\gamma_0 + e_F + e_0)$ we have $0<q_F<1$ since $e_F> -e_0$ (equivalent to $e>-E_*$). Thus we can repeat the above scheme with $q=q_F$ to complete the proof. 
\end{proof}
Because the relation $\psi_\perp =  (e^{T_*} -I)\phi_0$ holds, it is immediate that $\Vert \psi_\perp \Vert_{\mathcal H}$ is small if and only if $\Vert t_*\Vert_{\mathcal V}$ is. It is a fact that the operator norm
$\| T\|_{\mathcal{B}(\mathcal{H})}$ is equivalent to the norm
$\|t\|_{\mathcal{V}}$, see Ref.~\cite{Rohwedder2013b}.
We now state the first assumption:
\begin{assumption}
	\label{A1} Let $\eta_\varepsilon:= \gamma_\varepsilon c/(\gamma_\varepsilon+ e + E_* )$. We assume the following:
	\begin{itemize}
		\item[(a)] Eq. \eqref{gammaStar} holds with a strictly positive spectral gap $\gamma_*>0$. 
		\item[(b)] The optimal amplitudes $t_*$ and $\lambda_*$ are sufficiently small in $\Vert \cdot\Vert_{\mathcal V}$ norm. With
		$C_* := C+|E_*|$ we then assume
		$\Vert \psi_\perp \Vert_{\mathcal H}< \varepsilon$, where
		$\varepsilon >0$ is chosen such that
		\begin{equation}
		\begin{split}
		b_*(t_*,\lambda_*)&:=\Vert e^{-T_*^\dagger} e^{\Lambda_*} - I \Vert_{\mathcal B(\mathcal H)} + \Vert e^{-T_*^\dagger} e^{\Lambda_*} \Vert_{\mathcal B(\mathcal H)}
		\Vert e^{T_*} -I \Vert_{\mathcal B(\mathcal H)} \\
		& \quad +
		K \Vert \phi_0 \Vert_\mathcal{H} \Vert e^{ -T_*^\dagger}\Vert_{\mathcal B(\mathcal H)} \Vert e^{ T_*}\Vert_{\mathcal B(\mathcal H)}
		\Vert e^{ \Lambda_*} -I\Vert_{\mathcal B(\mathcal H)} < \frac {\eta_\varepsilon }{C_*}.
		\label{tlambdaCond}
		\end{split}
		\end{equation}
		Here, $K$ is a constant such that $\| T\|_{\mathcal B(\mathcal{H})} \leq K
		\|t\|_{\mathcal{V}}$, which exists since the norms are
		equivalent.
	\end{itemize}
\end{assumption}
\begin{remark}
	It is in fact possible to choose $\varepsilon >0$ such that
	\eqref{tlambdaCond} holds. Indeed, $\varepsilon=0$ is equivalent to $ t_*
	=  \lambda_* = 0$, and $b_*(t_*,\lambda_*) = b(\varepsilon)$, a smooth function of
	$\varepsilon$. Since, $b(\varepsilon) \to 0+$ as
	$\varepsilon \to 0+$ and $\gamma_\varepsilon$ tends to the spectral gap $\gamma_*$, there exists a $\varepsilon_0$ such that
	$b_* < \eta_\varepsilon/C_*$ for
	$\varepsilon\leq \varepsilon_0$. Furthermore, at $\varepsilon=0$ we have $\psi_* = \phi_0$, such that $\gamma_* = \gamma_0$ and $\mathcal P_* = \mathcal P$.
\end{remark}

We next define the similarity transformed Hamiltonian $ H_t$ and the doubly similarity transformed Hamiltonian $ H_{t,\lambda}$ as given by
\[
H_t := e^{- T}  H e^{ T} , \quad  H_{t,\lambda} := e^{ \Lambda^\dagger}  H_t e^{- \Lambda^\dagger}.
\]
Note that $(H_t)_\lambda \neq H_{t,\lambda}$. Since $e^{ T_*}\phi_0$ solves the SE with eigenvalue $E_*$, $\phi_0$ is an eigenfunction of $ H_{t_*}$ with the same eigenvalue. This fact and $e^{  \Lambda_*^\dagger}\phi_0 =e^{ - \Lambda_*^\dagger}\phi_0= \phi_0$ 
make it easy to verify (i) in
\begin{lemma}
	\label{SimSEs}
	Let $f(t_*) = \mathcal F(t_*,\lambda_*) =0$ and $E_* = \mathcal E(t_*,\lambda_*)$. Then
	
	(i) $ H_{t_*}\phi_0 = E_*\phi_0$ 
	and $ H_{t_*,\lambda_*} \phi_0 =   E_*\phi_0$.
	
	(ii) $H_{t_*,\lambda_*}^\dag\phi_0 = E_*\phi_0$.
\end{lemma}
\begin{proof}
	It remains to prove (ii). We know that (by definition of the left eigenfunction of $H$)
	\[
	\langle \phi_0, e^{ \Lambda_*^\dagger} e^{- T_*}  H = E_* \langle \phi_0,e^{ \Lambda_*^\dagger} e^{- T_*}.
	\]
	Thus $\langle \phi_0, e^{ \Lambda_*^\dagger}   H_{t_*}$ equals $E_*\langle \phi_0, e^{ \Lambda_*^\dagger}$, i.e., $ H_{t_*}^\dagger e^{ \Lambda_*} \phi_0\rangle = E_*e^{ \Lambda_*} \phi_0\rangle$.
\end{proof}
\begin{remark}
	Note that Lemma \ref{SimSEs} is valid for any critical point $(t_c,\lambda_c)$ with corresponding eigenvalue $E_c$, not only the ground state ($(t_*,\lambda_*)$ and $E_*$). 
	Furthermore, as stated in Lemma \ref{SimSEs}, the double similarity transform makes $\phi_0$ both the
		left and right eigenvector of $H_{t_*,\lambda_*}$ with the same eigenvalue.
\end{remark}

We now move on to Assumption~\ref{A2}, which corresponds to an
assumption made in Ref.~\cite{Rohwedder2013b}, but suitable for ECC. 
Roughly speaking, instead of assuming that the reference $\phi_0$ is
sufficiently accurate, in Assumption~\ref{A2} we assume that we have a
splitting $ H =  F + W$ where $ F$ is a one-body operator, and
where $ W$ is sufficiently small in some appropriate sense. For
example, $ F$ can be the Fock operator and $ W$ the fluctuation
potential of a molecule in the Born--Oppenheimer
approximation. Moreover, we assume that $ F\phi_0 = e_0\phi_0$ and
that \eqref{gamma0} holds, where $\gamma_0$ is the so-called HOMO-LUMO gap.

It can be remarked, that due to the structure of $H$, the
Baker--Campbell--Haus\-dorff (BCH) expansion for $H_t$ terminates
identically after four nested commutators in the case of a two-body
interaction operator, i.e., $H_t$ is actually a polynomial of low
order, independently of the number of particles. 

The expansion for the outer similarity
transform in $ H_{t,\lambda}$ also truncates, albeit at a higher
order. Thus, we have a finite sum
\begin{equation*}
H_{t,\lambda} = \sum_{m,n} \frac{1}{n!m!}
[[ H, T]_{(n)},- \Lambda^\dag]_{(m)}.
\end{equation*}
Here $[ A, B]_{(n)}$ denotes $ A$ $n$-fold commutated with $ B$ and
$[ A, B]_{(0)}:=  A$. 
For $(t,\lambda) \in\mathcal V \times \mathcal V$, we define the operator $ O(t,\lambda)$ through the relation
\begin{equation}
\label{Htlambda}
H_{t,\lambda}  = H  + [ F,  T] + [ \Lambda^\dagger, F]  + O(t,\lambda).
\end{equation}
The significance of $ O(t,\lambda)$ is that \eqref{Htlambda} implies 
\begin{equation}
\mathcal{E}(t,\lambda) - \langle \phi_0, H \phi_0\rangle =
\braket{\phi_0, O(t,\lambda)\phi_0}, \label{eq:corr-fun}
\end{equation}
i.e., $O(t,\lambda)$ gives all nontrivial contributions to
$\mathcal{E}$. In the Hartree--Fock case, the right-hand side of
Eq.~\eqref{eq:corr-fun} is the correlation energy functional,
since the Hartree--Fock energy is given by $E_\text{HF} =
\langle \phi_0, H\phi_0\rangle$.

The idea is that if 
the reference $\phi_0$ is sufficiently
good, the mapping $(t,\lambda)\mapsto  O(t,\lambda)$ will be
well-behaved. In fact, since $ O(t,\lambda)$ is a (Fr\'echet-)smooth map,
it is locally Lipschitz: Given
$(t,\lambda)\in\mathcal{V}\times\mathcal{V}$, there exist
$\delta,L>0$ such that for all $(t_i,\lambda_i) \in B_\delta(t,\lambda)$,
\begin{equation*}
\| O(t_1,\lambda_1) -
O(t_2,\lambda_2)\|_{\mathcal B(\mathcal{H},\mathcal{H}')} \leq L \|(t_1-t_2,\lambda_1-\lambda_2)\|_{\mathcal{V}\times\mathcal{V}}.
\end{equation*}
In our case, we assume that $L$ is sufficiently small at
$(t_*,\lambda_*)$. This, in a sense, measures the smallness of $ W$.
\begin{assumption}\label{A2}
	Let $H =  F +  W$ and $\eta_0 := \gamma_0 c_F/(\gamma_0 + e_F + e_0)$. We assume the following:
	\begin{itemize}
		\item[(a)]
		$F : \mathcal{H}\to\mathcal{H}'$ is a one-body operator that satisfies the same conditions as
		$H$, i.e., it is symmetric, bounded, and satisfies a G{\aa}rding
		estimate (with constants $e_F,c_F$), as in
		Eqs.~(\ref{eq:bounded-operator}--\ref{eq:symmetric-operator}). The constant that bounds $F$ is denoted 
		$C_F$ and we set $C_0 := C_F + |e_0|$. 
		\item[(b)]
		$ F\phi_0 = e_0 \phi_0$ where
		$e_0$ is the smallest eigenvalue of $ F$. Eq. \eqref{gamma0} holds with a $\gamma_0>0$, i.e., there is a strictly positive
		HOMO-LUMO gap. 
		In particular, Lemma \ref{LemmaRoh} gives that \eqref{fock1} holds for all $\psi\in\mathcal Q$.
		\item[(c)]
		The Lipschitz constant $L$ at $(t_*,\lambda_*)$ and
		$\|\lambda_*\|_{\mathcal{V}}$ are not too large, so that that the
		following inequality holds:
		\begin{equation}
		\begin{split}
		0 <  \gamma &:= \eta_0 - \frac{1}{2}L\|\phi_0\|_{\mathcal{H}}\big(3 + K
		\|(e^{ \Lambda_*}-1)\phi_0\|_{\mathcal{H}} +
		\|e^{ \Lambda_*}\phi_0\|_{\mathcal{H}}/\|\phi_0\|_{\mathcal{H}}
		\\ & \quad+ 2\|e^{ \Lambda_*}\|_{\mathcal B(\mathcal{H})}\big) -C_0 \|e^{ \Lambda_*}-1\|_{\mathcal B(\mathcal{H})}.
		\label{eq:gammadef}
		\end{split}
		\end{equation}
		Here, $K$ is a constant such that $\| T\|_{\mathcal B(\mathcal{H})} \leq K
		\|t\|_{\mathcal{V}}$, which exists since the norms are
		equivalent. 
	\end{itemize}
\end{assumption}
\begin{remark}
	Assumption~\ref{A2}(c) does not assume that $\lambda_*$ is small
	compared to $\lambda_1-\lambda_2$. However,
	$\lambda_*$ (and $L$) cannot be too large, since then $\gamma$ eventually
	becomes negative. If we do assume that
	$\|\lambda_*\|_{\mathcal{V}} < \delta$, we obtain some
        simplifications, see Corollary~\ref{thm:tiny} below.
\end{remark}

\subsection{Proof of Monotonicity}
We set $\Delta :=\Delta_1+\Delta_2$, the left-hand side of Eq.~\eqref{Delta1212}. We then wish to prove
\begin{equation}
\label{Pf1}
\Delta \geq \gamma \big( \Vert t_1-t_2 \Vert_{\mathcal V}^2 + \Vert \lambda_1-\lambda_2 \Vert_{\mathcal V}^2\big)
\end{equation}
where $(t_i,\lambda_i ) \in B_\delta(t_*,\lambda_*)$ and $\gamma,\delta>0$.  To simplify notation we define ${\bar{T}} = ( T_1
+  T_2)/2$ and $\delta  T =  T_1 -  T_2$, and similarly ${ \bar{\Lambda}} = ( \Lambda_1 +  \Lambda_2)/2$ and
$\delta \Lambda = \Lambda_1 - \Lambda_2$. Consequently, we write $\Vert \delta t \Vert_{\mathcal V}$ and $\Vert \delta \lambda \Vert_{\mathcal V}$ for $\Vert t_1-t_2 \Vert_{\mathcal V}$ and $\Vert \lambda_1-\lambda_2 \Vert_{\mathcal V}$, respectively. 
\begin{theorem}\label{thm:taylormono}
	Assume that Assumption \ref{A1} holds. Then $\mathcal F$ is
	strongly monotone locally at $(t_*,\lambda_*)$, 
	$\mathcal{F}(t_*,\lambda_*)=0$, belonging to the ground-state
	energy $E_* = \mathcal{E}(t_*,\lambda_*)$. 	
\end{theorem}
\begin{proof}
	Using the formulas~\eqref{Feq} for the
	partial derivatives, we obtain for the two terms in
	Eq.~\eqref{Delta1212},
	\begin{align*}
	\Delta_1 &= \langle \delta  T \phi_0, \big(e^{ \Lambda_1^\dagger} H_{t_1} - e^{ \Lambda_2^\dagger} H_{t_2} \big) \phi_0\rangle,\\
	\Delta_2 &= \langle \phi_0,\big( e^{ \Lambda_1^\dagger}[ H_{t_1},\delta  \Lambda] - 
	e^{ \Lambda_2^\dagger}[ H_{t_2},\delta \Lambda]\big)\phi_0 \rangle.
	\end{align*}
	Moreover, we make use of the following notation 
	$g_i := t_i - t_*$, $k_i := \lambda_i - \lambda_*$ and define the excitation operators
	$ G_i := \sum_{\mu }(g_i)_\mu  X_\mu$ and $ K_i := \sum_{\mu }(k_i)_\mu  X_\mu$. Also we write $\delta  G$ and $\delta  K$ as for $ T$ and $ \Lambda$, where of course $\delta G = \delta T$ and $\delta K = \delta \Lambda$. 
	As in \cite{Rohwedder2013b}, we note that the similarity transformed Hamiltonians $ H_{t_i}$ can be expanded in terms of $ H_{t_*}$ as
	\begin{equation}
	H_{t_i} = H_{t_*} + [ H_{t_*},  G_i] + \mathcal{O}(\Vert g_i\Vert_{\mathcal V}^2).
	\label{HstarTaylor}
	\end{equation}

	Let $\tilde \Delta$ be the second-order Taylor expansion of
	$\Delta$ around $(t_*,\lambda_*)$, i.e.,
	$
	\Delta = \tilde{\Delta}  + \mathcal O ( \Vert (\delta t ,\delta \lambda) \Vert_{\mathcal V \times \mathcal V}^3).
	$
	We will demonstrate the claim by first showing that $\tilde \Delta$ satisfies \eqref{Pf1} for some $\tilde \gamma>0$, using Assumption \ref{A1}.
	Now by \eqref{HstarTaylor} and $ \Lambda_i =    K_i+  \Lambda_*$, we see that
	\begin{align*}
	\Delta_1 &= \langle \delta  T\phi_0, \big( e^{ K_1^\dagger}e^{ \Lambda_*^\dagger}( H_{t_*}  +  [ H_{t_*},  G_1] +\mathcal{O}(\Vert g_1\Vert_{\mathcal V}^2) )\\
	&\quad - e^{ K_2^\dagger }e^{ \Lambda_*^\dagger}( H_{t_*}  +  [ H_{t_*},  G_2] + \mathcal{O}(\Vert g_2\Vert_{\mathcal V}^2))\big) \phi_0\rangle.
	\end{align*}
	With the aid of Lemma \ref{SimSEs} and since $ e^{ K_i^\dagger} \phi_0 =\phi_0$, it holds
	\begin{align*}
	&\Delta_1 = \langle \delta  T \phi_0, \big( e^{
          K_1^\dagger}e^{ \Lambda_*^\dagger}  [ H_{t_*},  G_1] - e^{
          K_2^\dagger }e^{ \Lambda_*^\dagger}  [ H_{t_*}, G_2] +
          \mathcal{O}(\Vert g_1\Vert_{\mathcal V}^2) +
          \mathcal{O}(\Vert g_2\Vert_{\mathcal V}^2)  \big) \phi_0\rangle.
	\end{align*}
	As a next step we truncate $e^{ K_i^\dagger} =  I + \mathcal O(\Vert k_i \Vert_{\mathcal V})$ and there holds
	\begin{align*}
	\Delta_1 
	&= \langle \delta T \phi_0, e^{ \Lambda_*^\dagger}[ H_{t_*},\delta  T]\phi_0\rangle
	+  \sum_{k=0}^3 \mathcal{O}(\Vert g_i\Vert_{\mathcal V}^k  \Vert k_i\Vert_{\mathcal V}^{3-k})\\
	&=\langle \delta T\phi_0, e^{\Lambda_*^\dagger}( H_{t_*} - E_*)\delta T\phi_0\rangle 
	+  \sum_{k=0}^3 \mathcal{O}(\Vert g_i\Vert_{\mathcal V}^k  \Vert k_i\Vert_{\mathcal V}^{3-k}). 
	\end{align*}
	Again we have made use of Lemma
        \ref{SimSEs}. Equation~\eqref{EqRoh} from Lemma
        \ref{LemmaRoh} and \eqref{eq:bounded-operator} give two useful bounds,
        \begin{align}
          \braket{\psi',(H-E_*)\psi} &\geq \eta_\varepsilon
          \|\psi\|^2_{\mathcal{H}} - C_* \|\psi'-\psi\|_{\mathcal{H}}
                                       \|\psi\|_{\mathcal{H}}, \label{eq:b1}
          \\
          \braket{\psi',(H-E_*)\psi} &\geq -C_*
                                       \|\psi'\|_{\mathcal{H}}\|\psi\|_{\mathcal{H}}. \label{eq:b2}
        \end{align}
        Using these,
	\begin{align*}
	\tilde \Delta_1 
	&= \langle \delta T\phi_0, e^{\Lambda_*^\dagger}( H_{t_*} - E_*)\delta T\phi_0\rangle \\
	& = \langle e^{-T_*^\dagger} e^{\Lambda_*} \delta T \phi_0, (H-E_*) \delta T \phi_0 \rangle + 
	\langle e^{-T_*^\dagger} e^{\Lambda_*} \delta T \phi_0, (H-E_*)(e^{T_*}-I) \delta T \phi_0 \rangle\\
	& \geq \eta_\varepsilon \Vert \delta T \phi_0 \Vert_{\mathcal H}^2 - C_* \Vert e^{-T_*^\dagger} e^{\Lambda_*} - I \Vert_{\mathcal B(\mathcal H)}\Vert \delta T \phi_0 \Vert_{\mathcal H}^2  \\
	& \quad - C_*\Vert e^{-T_*^\dagger} e^{\Lambda_*} \Vert_{\mathcal B(\mathcal H)}
	\Vert e^{T_*} -I \Vert_{\mathcal B(\mathcal H)}\Vert \delta T \phi_0 \Vert_{\mathcal H}^2\\
	&= \Vert \delta t \Vert_{\mathcal V}^2 \big(  \eta_\varepsilon - C_* (\Vert e^{-T_*^\dagger} e^{\Lambda_*} - I \Vert_{\mathcal B(\mathcal H)} + \Vert e^{-T_*^\dagger} e^{\Lambda_*} \Vert_{\mathcal B(\mathcal H)}
	\Vert e^{T_*} -I \Vert_{\mathcal B(\mathcal H)}  ) \big).
	\end{align*}

	Next, we look at $\Delta_2$. Proceeding in similar a fashion, we compute
	\begin{equation}
	\begin{split}
	\Delta_2 
	&= \langle \phi_0,( I +   K_1^\dagger + \mathcal O( \Vert k_1\Vert_{\mathcal V}^2  ))e^{ \Lambda_*^\dagger} 
	[ H_{t_*} + [ H_{t_*},  G_1] + \mathcal O ( \Vert g_1 \Vert_{\mathcal V}^2),\delta  \Lambda ] \phi_0\rangle  \\
	&\quad -\langle \phi_0,( I +   K_2^\dagger + \mathcal O( \Vert k_2\Vert_{\mathcal V}^2  ))e^{ \Lambda_*^\dagger} 
	[ H_{t_*} + [ H_{t_*},  G_2]+ \mathcal O ( \Vert g_2 \Vert_{\mathcal V}^2) ,\delta \Lambda] \phi_0\rangle  \\
	&= \langle \phi_0,  \delta \Lambda^\dagger e^{\Lambda_*^\dagger}( H_{t_*} -E_*)\delta \Lambda \phi_0\rangle 
	+ \langle \phi_0, e^{\Lambda_*^\dagger}\big[  [ H_{t_*},  \delta  T],  \delta  \Lambda \big]\phi_0\rangle \\
	& \quad + \sum_{k=0}^3 \mathcal{O}(\Vert g_i\Vert_{\mathcal V}^k  \Vert k_i\Vert_{\mathcal V}^{3-k})\\
	&=: \tilde \Delta_{2,1} +\tilde  \Delta_{2,2} +\sum_{k=0}^3 \mathcal{O}(\Vert g_i\Vert_{\mathcal V}^k  \Vert k_i\Vert_{\mathcal V}^{3-k}),
	\label{Pf2}
	\end{split}
	\end{equation}
	where the last equality defines $\tilde \Delta_{2,1}$ and $\tilde \Delta_{2,2}$. 
	For $\tilde \Delta_{2,1}$ in \eqref{Pf2}, we again employ
        Eqs.~\eqref{eq:b1} and \eqref{eq:b2} to obtain 
	\begin{align*}
	\tilde \Delta_{2,1} &=\langle \phi_0, \delta  \Lambda^\dagger e^{ \Lambda_*^\dagger}( H_{t_*} -E_*)\delta  \Lambda \phi_0\rangle \\
	&= \langle e^{-T_*^\dagger } e^{\Lambda_*} \delta  \Lambda \phi_0, ( H-E_*) \delta \Lambda \phi_0 \rangle +
	\langle e^{-T_*^\dagger } e^{\Lambda_*} \delta  \Lambda \phi_0, ( H-E_*)(e^{T_*} - I) \delta \Lambda \phi_0 \rangle \\
	& \geq \eta_\varepsilon \Vert \delta \Lambda \phi_0 \Vert_{\mathcal H}^2 - C_* \Vert e^{-T_*^\dagger} e^{\Lambda_*} - I \Vert_{\mathcal B(\mathcal H)}\Vert \delta \Lambda \phi_0 \Vert_{\mathcal H}^2  \\
	& \quad - C_*\Vert e^{-T_*^\dagger} e^{\Lambda_*} \Vert_{\mathcal B(\mathcal H)}
	\Vert e^{T_*} -I \Vert_{\mathcal B(\mathcal H)}\Vert \delta \Lambda \phi_0 \Vert_{\mathcal H}^2\\
	&= \Vert \delta \lambda  \Vert_{\mathcal V}^2 \big(  \eta_\varepsilon - C_* (\Vert e^{-T_*^\dagger} e^{\Lambda_*} - I \Vert_{\mathcal B(\mathcal H)} + \Vert e^{-T_*^\dagger} e^{\Lambda_*} \Vert_{\mathcal B(\mathcal H)}
	\Vert e^{T_*} -I \Vert_{\mathcal B(\mathcal H)}  ) \big).
	\end{align*} 
	Turning to $ \tilde \Delta_{2,2}$ in \eqref{Pf2}, we have by Lemma \ref{SimSEs}
	\begin{align*}
	\tilde \Delta_{2,2} &=    \langle \phi_0, e^{ \Lambda_*^\dagger}\big[  [ H_{t_*}, \delta  T],  \delta  \Lambda \big]\phi_0\rangle\\
	&= \langle e^{ \Lambda_*}\phi_0, \big( ( H_{t_*}\delta T - \delta T  H_{t_*} ) \delta  \Lambda   
	- \delta\Lambda  ( H_{t_*} \delta  T -	\delta  T  H_{t_*} ) \big) \phi_0\rangle\\
	&= \langle e^{ \Lambda_*}\phi_0, \big( \delta  T (E_* -  H_{t_*} )\delta \Lambda   
	- \delta \Lambda  ( H_{t_*} - E_*)\delta T  \big) \phi_0\rangle.
	\end{align*}
	Since 
	\[
	\big(\delta  T (E_* -  H_{t_*} )\delta \Lambda   
	- \delta \Lambda  ( H_{t_*} - E_*)\delta T \big) \phi_0 \in \mathcal Q,
	\]
	we only need to keep that part of $e^{ \Lambda_*}\phi_0$ that belongs to $\mathcal Q$. 
	Using Eq.~\eqref{eq:b2}, it holds that
	\begin{align*}
	\tilde \Delta_{2,2} &=  \langle e^{-T_*^\dagger }\delta  T^\dagger (e^{\Lambda_*}-I) \phi_0,  (E_* -  H ) e^{T_*}\delta \Lambda \phi_0 \rangle   \\ 
	& \quad +\langle e^{-T_*^\dagger }\delta \Lambda^\dagger (e^{\Lambda_*}-I) \phi_0,   (E_* -  H) e^{T_*} \delta T  \phi_0\rangle\\
	&\geq- 2C_*K \Vert \phi_0 \Vert_\mathcal{H} \Vert e^{ -T_*^\dagger}\Vert_{\mathcal B(\mathcal H)} \Vert e^{ T_*}\Vert_{\mathcal B(\mathcal H)}
	\Vert e^{ \Lambda_*} -I\Vert_{\mathcal B(\mathcal H)} 	\Vert \delta t \Vert_{\mathcal V} \Vert \delta  \lambda  \Vert_{\mathcal V} \\
	&\geq- C_*K \Vert \phi_0 \Vert_\mathcal{H} \Vert e^{ -T_*^\dagger}\Vert_{\mathcal B(\mathcal H)} \Vert e^{ T_*}\Vert_{\mathcal B(\mathcal H)}
	\Vert e^{ \Lambda_*} -I\Vert_{\mathcal B(\mathcal H)} 
	\big( \Vert \delta \lambda \Vert_{\mathcal V}^2 
	+ \Vert \delta  t  \Vert_{\mathcal V}^2  \big).
	\end{align*}

	To summarize, collecting the lower bounds for $\tilde \Delta_1$ and $\tilde \Delta_{2,i}$ we can now conclude by means of the definition given by \eqref{tlambdaCond}
	\begin{align*}
	\tilde{\Delta} \geq
	(\eta_\varepsilon - C_*b_* (t_*,\lambda_*) )  \big(  \Vert \delta t \Vert_\mathcal{V}^2+\Vert \delta \lambda \Vert_\mathcal{V}^2 \big).
	\end{align*}
	By Assumption 1, $\tilde \gamma := \eta_\varepsilon - C_*b_*(t_*,\lambda_*) >0$ such that 
	\begin{align}
	\tilde \Delta \geq  \tilde \gamma \big(  \Vert \delta t \Vert_\mathcal{V}^2+\Vert\delta \lambda \Vert_\mathcal{V}^2 \big),\quad \tilde \gamma>0,
	\label{tildeDelta}
	\end{align}
	holds. To conclude the proof, we just have to note that by \eqref{tildeDelta}
	\[
	\Delta \geq \tilde \gamma \big(  \Vert \delta t \Vert_\mathcal{V}^2+\Vert\delta \lambda \Vert_\mathcal{V}^2 \big) 
	+ \mathcal O ( \Vert (\delta t ,\delta \lambda) \Vert_{\mathcal V \times \mathcal V}^3) 
	\]  
	and by choosing $\delta$ sufficiently small there holds for some $\gamma\in(0,\tilde \gamma]$
	\[
	\Delta \geq  \gamma \big(  \Vert \delta t \Vert_\mathcal{V}^2+\Vert \delta \lambda \Vert_\mathcal{V}^2 \big) 
	\]
	for $(t_i,\lambda_i) \in B_{\delta}(t_*,\lambda_*)$.
\end{proof}

\begin{theorem}\label{thm:splittmono}
	Assume that Assumption \ref{A2} holds. Then $\mathcal F$ is
	strongly monotone locally at $(t_*,\lambda_*)$, 
	$\mathcal{F}(t_*,\lambda_*)=0$, belonging to the ground-state
	energy $E_* = \mathcal{E}(t_*,\lambda_*)$.
	\label{ThmHsplitt}
\end{theorem}
\begin{proof}
	As in the proof of Theorem \ref{thm:taylormono}, we study
        $\Delta_1$ and $\Delta_2$ separately before adding them
        together. We begin by noting that 
	\[
	\Delta_1 = \braket{\delta T\phi_0,( e^{\Lambda_1^\dagger}   H_{t_1} -
		e^{\Lambda_2^\dagger}   H_{t_2})\phi_0 } =
	\braket{\delta T\phi_0, ( H_{t_1,\lambda_1} -
		H_{t_2,\lambda_2})\phi_0},
	\] 
	because any de-excitation of the reference $\phi_0$ gives zero identically. Now, using Assumption~\ref{A2} 
	and the definition~\eqref{Htlambda} of the operator $O(t,\lambda)$ 	
	we immediately obtain the following lower bound for $\Delta_1$,
	\begin{equation*}
	\begin{split}
	\Delta_1 &= \braket{\delta T\phi_0, ( H_{t_1,\lambda_1} -
		H_{t_2,\lambda_2})\phi_0} \\
	& = \braket{\delta  T\phi_0, \big([ F, \delta  T] + [\delta \Lambda^\dag, F]
		+  O(t_1,\lambda_1) -  O(t_2,\lambda_2)\big) \phi_0} \\
	& = \braket{\delta T\phi_0, ( F-e_0)\delta  T\phi_0} +
	\braket{\delta  T\phi_0, ( O(t_1,\lambda_1) -
		O(t_2,\lambda_2))\phi_0} \\
	& \geq \eta_0 \|\delta  T\phi_0\|_{\mathcal H}^2 - L \|\delta  T\phi_0\|_{\mathcal H}\|(\delta
	t, \delta\lambda)\|_{\mathcal V \times \mathcal V} \|\phi_0\|_{\mathcal H} \\ 
	& = \eta_0 \|\delta t\|_{\mathcal V}^2 - L\|\phi_0\|_{\mathcal H} \|\delta t\|_{\mathcal V} (\|\delta
	t\|_{\mathcal V}^2 + \|\delta\lambda\|_{\mathcal V}^2)^{1/2} \\
	& \geq \eta_0 \|\delta t\|_{\mathcal V}^2 - L \|\phi_0\|_{\mathcal H}\|\delta t\|_{\mathcal V}(\|\delta
	t\|_{\mathcal V} + \|\delta\lambda\|_{\mathcal V}) .
	\end{split} \label{eq:Delta1}
	\end{equation*}

	We next turn to $\Delta_2$. It holds,
	\begin{equation}
	e^{\Lambda_1} - e^{\Lambda_2} = e^{{\bar{\Lambda}}} \delta \Lambda +
	\mathcal O(\|\delta\lambda\|_{\mathcal V}^2).
	\end{equation}
        We compute	
	\begin{equation}
	\begin{split}
	\Delta_2 &= \braket{\phi_0, \big(e^{ \Lambda_1^\dag}
		[ H_{t_1},\delta  \Lambda] - e^{ \Lambda_2^\dag}
		[ H_{t_2}\delta  \Lambda] \big)\phi_0} \\
	& = \braket{\phi_0, \big( (e^{ \Lambda_1^\dagger} -
		e^{ \Lambda_2^\dagger})[ H_{t_1},\delta  \Lambda] + e^{ \Lambda_2^\dag}[ H_{t_1} -
		H_{t_2}, \delta \Lambda]\big)\phi_0} \\
	& = \braket{\phi_0, \big( e^{{\bar{\Lambda}}^\dag}\delta\Lambda^\dag [ F + W+
		O(t_1,0),\delta \Lambda] - e^{{\bar{\Lambda}}^\dag}[ O(t_1,0) -
		O(t_2,0),\delta\Lambda] \big) \phi_0} \\
	&\quad+ \mathcal O(\|\delta\lambda\|_{\mathcal V}^3) +
	\mathcal O(\|\delta\lambda\|_{\mathcal V}\|\delta t\|_{\mathcal V}^2) + \mathcal O(\|\delta\lambda \|_{\mathcal V}^2 \|\delta t\|_{\mathcal V}).
	\end{split}\label{eq:step0}
	\end{equation}
	In the last equality, we exploited that the second-order nested
	commutator of $ F$ with two excitation operators
	vanishes. This so since for $\mu\neq 0$ we have that
	$[F,X_\mu]$ is an excitation operator and consequently $[[F,T],T']=0$. Moreover, we
	used that $ O(t_1,0)-  O(t_2,0) = \mathcal O(\|\delta t\|)$, allowing us to replace
	$\Lambda_2$ with ${\bar{\Lambda}} = \Lambda_2 +
	\delta\Lambda/2$, a change which only affects the higher-order terms.

	Define $\tilde{\Delta}_2$ as the leading second-order
	term of $\Delta_2$, i.e., the first term in the last line of Eq.~\eqref{eq:step0}, neglecting the
	third-order remainders (note that these are in total $\mathcal O(\|(\delta
	t,\delta\lambda)\|_{\mathcal V \times \mathcal V}^3)$). We will start by finding $\tilde\gamma > 0$ such that
	\begin{equation*}
	{\Delta}_1 + \tilde{\Delta}_2 \geq \tilde\gamma (\|\delta t\|_{\mathcal V}^2 +
	\|\delta\lambda\|_{\mathcal V}^2).
	\end{equation*}
	We split $\tilde{\Delta}_2$ into two contributions, $\tilde \Delta_{2,i}$,
	$i=1,2$.

	Since $O(t,0) + W= e^{-T}We^{T}$, the BCH formula gives,
	\begin{equation*}
	O(t + \delta\lambda,0) - O(t,0) = [ O(t,0) + W,\delta\Lambda]
	+ \mathcal O(\|\delta\lambda\|^2).
      \end{equation*}
      This gives us the directional derivative of $O(\cdot,0)$ in the
        direction $\delta \lambda$, 
	\[
	D  O(t,0)(\delta\lambda) = [O(t,0) + W,\delta\Lambda].
	\]
	On the other hand, $O$ is Lipschitz, so
	that
	\begin{equation*}
	\|[ O(t_1,0) + W,\delta  \Lambda]\|_{\mathcal B(\mathcal{H},\mathcal{H}')} \leq
	\|D  O(t_1,0)\|_{\mathcal B(\mathcal{V},B(\mathcal{H},\mathcal{H}'))}
	\|\delta\lambda\|_{\mathcal{V}} \leq ( L + K'\delta) \|\delta\lambda\|_{\mathcal{V}},
	\end{equation*}
	for some constant $K'$.
	
        A useful bound is obtained from Eq.~\eqref{fock1} from
        Lemma~\ref{LemmaRoh},
        \begin{equation}
          \braket{\psi',(F-e_0)\psi} \geq \eta_0
          \|\psi\|^2_{\mathcal{H}} - C_0 \|\psi'-\psi\|_{\mathcal{H}}
          \|\psi\|_{\mathcal{H}}. \label{eq:b3}
        \end{equation}
        The first contribution becomes
	\begin{equation*}
	\begin{split}
	\tilde \Delta_{2,1} &= \braket{\phi_0, e^{{\bar{\Lambda}}^\dag}
		\delta \Lambda^\dag [ F , \delta  \Lambda]\phi_0} + \braket{\phi_0, e^{{\bar{\Lambda}}^\dag}
		\delta \Lambda^\dag [  O(t_1,0)+  W, \delta \Lambda]\phi_0} \\
	& = \braket{e^{{\bar{\Lambda}}}\delta \Lambda\phi_0,
		( F-e_0)\delta \Lambda\phi_0}  + \braket{\delta \Lambda e^{{\bar{\Lambda}}}\phi_0,  [  O(t_1,0) + W, \delta \Lambda]\phi_0} \\
	& \geq \eta_0 \|\delta \Lambda\phi_0\|_{\mathcal H}^2 -
	C_0\|(e^{{\bar{\Lambda}}}-1)\delta \Lambda\phi_0\|_{\mathcal H}\|\delta \Lambda\phi_0\|_{\mathcal H}
	\\
	& \qquad -
	\|e^{{\bar{\Lambda}}}\delta\Lambda \phi_0\|_{\mathcal{H}} (L +
	K'\delta) \|\delta\lambda\|_{\mathcal{V}} \|\phi_0\|_{\mathcal{H}}\\
	& \geq \big( \eta_0  - C_0\|e^{{\bar{\Lambda}}}-1\|_{\mathcal B(\mathcal H)} - (L + K'\delta)\|\phi_0\|_{\mathcal{H}}
	\|e^{{\bar{\Lambda}}}\|_{\mathcal B(\mathcal{H})}\big) \|\delta\lambda\|^2_{\mathcal{V}}.
	\end{split}
	\end{equation*}
	The second contribution is
	\begin{equation*}
	\begin{split}
	\tilde \Delta_{2,2} &= \braket{\phi_0, e^{{\bar{\Lambda}}^\dag} [ O(t_1,0) -
		O(t_2,0),\delta\Lambda] \phi_0} \\ 
	&= \braket{e^{{\bar{\Lambda}}}\phi_0,
		( O(t_1,0)- O(t_2,0))\delta \Lambda \phi_0}- \braket{e^{{\bar{\Lambda}}}\phi_0,
		\delta\Lambda ( O(t_1,0)- O(t_2,0))\phi_0 }  \\
	& \geq  - L \|e^{{\bar{\Lambda}}}\phi_0\|_{\mathcal H}\|\delta\lambda\|_{\mathcal V}\|\delta t\|_{\mathcal V}
	-
	L\|\delta\Lambda^\dag(e^{{\bar{\Lambda}}}-1)\phi_0\|_{\mathcal H}\|\phi_0\|_{\mathcal H}\|\delta
	t\|_{\mathcal V}\\
	& \geq  - L \|e^{{\bar{\Lambda}}}\phi_0\|_{\mathcal H}\|\delta\lambda\|_{\mathcal V}\|\delta t\|_{\mathcal V}
	-LK\|(e^{{\bar{\Lambda}}}-1)\phi_0\|_{\mathcal H}\|\phi_0\|_{\mathcal H}\|\delta\lambda\|_{\mathcal V}\|\delta
	t\|_{\mathcal V} \\
	& = -L (K\|(e^{{\bar{\Lambda}}}-1)\phi_0\|_{\mathcal H} + \|e^{{\bar{\Lambda}}}\phi_0\|_{\mathcal H}/\|\phi_0\|_{\mathcal H}) \|\phi_0\|_{\mathcal H}\|\delta\lambda\|_{\mathcal V}\|\delta t\|_{\mathcal V}.
	\end{split}
	\end{equation*}
	We gather and obtain, 
	\begin{equation*}
	\begin{split}
	\Delta_1 + \tilde{\Delta}_2  &\geq \eta_0\|\delta t\|_{\mathcal V}^2 -
	L\|\phi_0\|_{\mathcal H}\|\delta t\|_{\mathcal V}(\|\delta t\|_{\mathcal V} + \|\delta\lambda\|_{\mathcal V})
	\\
	& \quad + \big( \eta_0
	- \| F-e_0\|_{\mathcal B (\mathcal H,\mathcal H')}\|e^{{\bar{\Lambda}}}-1\|_{\mathcal B (\mathcal H)} - (L + K'\delta)\|\phi_0\|_{\mathcal{H}}\|e^{{\bar{\Lambda}}}\|_{\mathcal B(\mathcal{H})}\big)\|\delta\lambda\|_{\mathcal V}^2 \\
	& \quad - L(K\|(e^{{\bar{\Lambda}}}-1)\phi_0\|_{\mathcal H} +
	\|e^{{\bar{\Lambda}}}\phi_0\|_{\mathcal H}/\|\phi_0\|_{\mathcal H}) \|\phi_0\|_{\mathcal H}\|\delta\lambda\|_{\mathcal V}\|\delta t\|_{\mathcal V} \\
	& \geq \big(\eta_0 - \frac{1}{2}L\|\phi_0\|_{\mathcal H}\big(3 + K
	\|(e^{{\bar{\Lambda}}}-1)\phi_0\|_{\mathcal H} +
	\|e^{{\bar{\Lambda}}}\phi_0\|_{\mathcal H}/\|\phi_0\|_{\mathcal H} + 2\|e^{{\bar{\Lambda}}}\|_{\mathcal B (\mathcal H)}\big)\\
	& 
	\quad -\| F-e_0\|_{\mathcal B ( \mathcal H, \mathcal H')}\|e^{{\bar{\Lambda}}}-1\|_{\mathcal B (\mathcal H)}   
	\big) \|(\delta t, \delta\lambda)\|_{\mathcal V \times \mathcal V}^2 
	\\
	&\quad - K'\delta \|\phi_0\|_{\mathcal H}
	\|e^{{\bar{\Lambda}}}\|_{\mathcal B (\mathcal H)}\|(\delta t, \delta\lambda)\|_{\mathcal V \times \mathcal V}^2 \\
	&  = : \tilde{\gamma}(\bar{t},\bar{\lambda}) \|(\delta t, \delta\lambda)\|_{\mathcal V \times \mathcal V}^2.
	\end{split}
	\end{equation*}
	We now note that, by Taylor's Theorem, $\tilde{\gamma}(\bar{t},\bar{\lambda}) =
	\gamma + \varepsilon(\bar{t},\bar{\lambda}) - K'\delta$, with $\gamma =
	\tilde{\gamma}(t_*,\lambda_*)>0$ by Eq.~\eqref{eq:gammadef} in Assumption \ref{A2},
	and $|\varepsilon| \leq C\delta$ for some $C\geq 0$. Thus, 
	\begin{equation*}
	\Delta_1 + \tilde{\Delta}_2 \geq
	(\gamma - (C+K')\delta)\|(\delta t, \delta\lambda)\|_{\mathcal V \times \mathcal V}^2. 
	\end{equation*}
	Finally,
	\begin{equation*}
	\Delta_1 + \Delta_2 \geq   (\gamma - (C+K')\delta)\|(\delta t,
	\delta\lambda)\|_{\mathcal V \times \mathcal V}^2 +\mathcal  O(\|(\delta t,\delta\lambda)\|_{\mathcal V \times \mathcal V}^3).
	\end{equation*}
	Since the third-order term cannot beat the second order term, by
	shrinking $\delta$, we get
	\begin{equation*}
	\Delta_1 + \Delta_2 \geq   (\gamma - (C+K')\delta')\|(t_1-t_2,
	\lambda_1-\lambda_2)\|_{\mathcal V \times \mathcal V}^2
	\end{equation*}
	whenever $(t_i, \lambda_i) \in B_{\delta'}(t_*,\lambda_*)$.
\end{proof}

\begin{corollary}\label{thm:tiny}
	Assume Assumption~\ref{A2}(a--b) holds, and additionally that we have 
	$\|\lambda_*\|_{\mathcal{V}}<\delta$. Also, assume that
	\begin{equation}
	0 <  \eta_0 - 3L\|\phi_0\|_{\mathcal{H}} . \label{eq:gammatiny}
	\end{equation}
	Then $\mathcal{F}$ is locally strongly monotone at the root
	$(t_*,\lambda_*)$ belonging to the ground-state energy.
\end{corollary}
\begin{proof} It is enough to observe that we need to Taylor expand
	$\gamma = \tilde{\gamma}(t_*,\lambda_*)$ to zeroth order, i.e.,
	setting $\lambda_*=0$ in Eq.~\eqref{eq:gammadef}. The
	reader can readily verify that this gives Eq.~\eqref{eq:gammatiny}.
\end{proof}

\subsection{Existence, uniqueness, truncations and error estimates}

Having obtained sufficient conditions for $\mathcal{F}$ to be locally
strongly monotone at $(t_*,\lambda_*)$, we can now apply the local
version of Zarantonello's theorem, Theorem~\ref{fCCthm}, to obtain existence and
local uniqueness of solutions, also for truncated versions of ECC.

In our setting, a (family of) truncated amplitude spaces
$\mathcal{V}_d\times\mathcal{V}_d$ is such that if we let the
dimension $d\to +\infty$, we can approximate $(t_*,\lambda_*)$
arbitrarily well. Of course, the usual truncation scheme defined by
all excitations up to a given excitation level and additionally the
restriction to a finite set of virtual orbitals, conforms with
this. In the sequel it will be assumed that $\mathcal V_d$ is closed in $\mathcal V$. 

The truncated ECC functional is the restriction $\mathcal{E}_d :
\mathcal{V}_d \times \mathcal{V}_d \to \RR$ of $\mathcal{E}$, giving the critical point
problem $D \mathcal{E}_d=0$, i.e.,
\begin{equation*}
\text{find $(t_d,\lambda_d) \in \mathcal{V}_d\times\mathcal{V}_d$}	
\quad \text{such that} \quad \frac{\partial \mathcal{E}(t_d,\lambda_d)}{\partial t_\mu} =
\frac{\partial \mathcal{E}(t_d,\lambda_d)}{\partial \lambda_\mu} = 0,
\end{equation*}
where $t_\mu$ ($\lambda_\mu$) are the components of $t \in
\mathcal{V}_d$ ($\lambda \in \mathcal{V}_d$) in some arbitrary
orthonormal basis. Since the flipping map in Eq.~\eqref{eq:flipping} commutes with projection
onto $\mathcal{V}_d\times\mathcal{V}_d$, the truncated ECC equations
can be written $ \mathcal{F}_d(t_d,\lambda_d) = 0$.

While stated as a theorem, our main result is really a corollary of
Theorems~\ref{thm:taylormono} and~\ref{ThmHsplitt}, and an elementary
application of 
Theorem~\ref{fCCthm}. The only point to check is that $\mathcal{F}$
is locally Lipschitz. However, $\mathcal{F}$ is (in fact infinitely)
continuously differentiable in the Fr\'echet sense. Such functions are
always locally Lipschitz.

\begin{theorem}
	\label{UandE}
	Assume that Assumption \ref{A1} or \ref{A2} holds such that
	$\mathcal F$ is locally strong\-ly monotone (with constant
	$\gamma$) on $B_{\delta} (t_*,\lambda_*)$, for
	some $\delta > 0$. Here, $(t_*,\lambda_*)$ is the root of
	$\mathcal{F}$ belonging to the ground-state energy. Furthermore, let
	$L$ be the local Lipschitz constant of $\mathcal{F}$ at $(t_*,\lambda_*)$.
	
	\begin{enumerate}
		\item[(i)] 
		The solution $(t_*,\lambda_*)$ of the continuous ECC equation
		$D\mathcal{E}(t,\lambda) =0$ on $\mathcal V \times \mathcal V$
		is locally unique.
		\item[(ii)]
		For sufficiently large $d$, the projected ECC problem $D\mathcal{E}_d(t,\lambda)=0$ has a unique
		solution $(t_d,\lambda_d)$ in the neighborhood
		$B_\delta(t_*,\lambda_*)\cap (\mathcal V_d\times \mathcal
		V_d)$. 
		The truncated solution $(t_d,\lambda_d)$ satisfies the
		estimate
		\begin{align}
		\Vert (t_d,\lambda_d) - (t_*,\lambda_*)\Vert_{\mathcal V\times \mathcal V}
		\leq \frac{L}{\gamma} d(\mathcal V_d\times \mathcal V_d,(t_*,\lambda_*)).
		\label{Est}
		\end{align}
	\end{enumerate}
\end{theorem}
\begin{remark}
	(i) The local uniqueness is also a direct consequence of the assumption that the ground state is non-degenerate and
	Lemma~\ref{lemma:ECCpara}.
	
	(ii) By the definition of the norm on $\mathcal V \times \mathcal V$, \eqref{Est} implies
	\begin{equation}
	\label{Est2}
	\Vert t_d - t_*\Vert_{\mathcal V}^2 + \Vert \lambda_d - \lambda_*\Vert_{\mathcal V}^2
	\leq \frac{L^2}{\gamma^2} \big( d(\mathcal V_d,t_*)^2 + d(\mathcal V_d,\lambda_*)^2 \big),
	\end{equation}
	and furthermore that $(t_d,\lambda_d) \to (t_*,\lambda_*)$ as $d\to +\infty$.	
\end{remark}

Theorem~\ref{UandE} guarantees that for sufficiently large discrete
amplitude spaces $\mathcal{V}_d$, the ECC equations actually have locally
unique solutions that approximate the exact solution. However, we
do not yet know what ``sufficiently large'' means.

By slightly adapting the proof of Theorem 4.1 in
Ref.~\cite{Rohwedder2013b}, we can obtain a sufficient condition on
$\mathcal{V}_d$. This argument rests on Brouwer's fixed point theorem:
any continuous function of a closed ball in $\RR^n$ into itself has a
fixed point. Here, we employ 
a version of
this result \cite{Emmrich2004}.
\begin{lemma}\label{thm:brouwer}
	Equip $\RR^n$ with any norm $\|\cdot\|_n$, and let $B_R$
	be the closed ball of radius $R$ centered at $\vec{x}=0$. Let
	$h : B_{R} \to \mathbb R^n$ be continuous and assume that on the
	boundary of $B_R$, $\langle h(\vec x),\vec x\rangle = h(\vec x) \cdot \vec x\geq 0$. Then
	$h(\vec x)=0$ for some $\vec x\in B_R$.
\end{lemma}
\begin{proof}
	Assume that  $h \neq 0$
	everywhere. Then $f(\vec x) := - R h(\vec x)/ \Vert h(\vec x) \Vert_n$
	is continuous, mapping the ball into itself (in fact, onto its
	boundary). Therefore, $f$ has a fixed point, say $\vec{x}_0$, i.e.,
	$\vec{x}_0 = -R h(\vec{x}_0)/ \Vert h(\vec{x}_0) \Vert_n$. However, this
	gives the contradiction
	$0 < \vec{x}_0\cdot\vec{x}_0 =
	-R\braket{h(\vec{x}_0),\vec{x}_0}/\|\vec{x}_0\|_n \leq 0$. 
\end{proof}

Following \cite{Rohwedder2013b}, the idea is now to choose $h_d$ such
that $\mathcal{F}_d =0$ is equivalent to $h_d=0$ and use the above
argument. 
\begin{theorem}
	Let $\mathcal V_d$ be a finite-dimensional subspace of
	$\mathcal V$ and set
	\begin{equation}
	\kappa_d:= \min_{(t,\lambda)\in \mathcal V_d\times \mathcal V_d}\Vert (t,\lambda) - (t_*,\lambda_*)\Vert_{\mathcal V\times \mathcal V} = \Vert (t_m,\lambda_m) - (t_*,\lambda_*)\Vert_{\mathcal V\times \mathcal V}.
	\label{Condition1}
	\end{equation}
	Assume that $\kappa_d$ satisfies
	\begin{equation}
	\kappa_d \leq \frac{\delta\gamma}{\gamma + L},
	\label{Condition2}
	\end{equation}
	where $\gamma$ and $L$ are the monotonicity and Lipschitz constants, respectively, that hold on $B_\delta(t_*,\lambda_*)$. 
	Then the projected extended coupled-cluster problem
	$ \mathcal{F}_d(t,\lambda)=0$ has a unique solution
	$(t_d,\lambda_d)$ in the neighborhood
	$B_\delta(t_*,\lambda_*)\cap (\mathcal V_d\times \mathcal V_d)$.
	\label{ThmEst}
\end{theorem}
\begin{proof}
	Let $d:= \dim \mathcal V_d$ and $\{b_j\}_{j=1}^d$ be an orthonormal
	basis of $\mathcal V_d$. Define the continuous vector-valued
	function $h_d:\mathbb R^{2d}\to \mathbb R^{2d}$ by
	$h_d(\vec x)=h_d(\vec{v},\vec{w}) = (\vec y, \vec z)$, where
	\[
	y_j = \langle D_\lambda\mathcal{E}(t_m + v,\lambda_m +
	w),b_j\rangle,\quad z_j= \langle D_t\mathcal{E}(t_m + v,\lambda_m +
	w),b_j\rangle ,
	\]
	and $v=\sum_{j=1}^d v_j b_j$, $\vec{v}=(v_1,\dots,v_d)$, $w=\sum_{j=1}^d w_j b_j$, $\vec{w}=(w_1,\dots,w_d)$. 
	Let $\Vert (\vec{v},\vec{w}) \Vert_{2d} := \Vert (v,w)\Vert_{\mathcal V\times
		\mathcal V}$, a norm on $\RR^{2d}$ (a fact that can be easily checked). 
	By definition, $h_d=0$ is equivalent to $\mathcal{F}_d =0$.
	
	We now choose
	$R:= \delta-\kappa_d \geq \delta L/(\gamma +
	L)>0$ and note that
	$(\vec{v},\vec{w}) \in B_R(t_m,\lambda_m)$ implies $(v,w)\in B_\delta(t_*,\lambda_*)$. 
	For $\vec x$ that satisfies $\Vert \vec x\Vert_{2d}=R $, we have using monotonicity and Lipschitz continuity of $\mathcal F$,
	\begin{equation*}
	\begin{split}
	\langle h_d(\vec x),\vec x\rangle &= \sum_{j=1}^d (y_j v_j + z_j w_j )
	=  \langle  \mathcal F(t_m+ v,\lambda_m+ w),(v,w)    \rangle  \\
	&=	\langle  \mathcal F(t_m+ v,\lambda_m+ w) - \mathcal F(t_m,\lambda_m),(v,w)    \rangle \\
	& \quad +\langle \mathcal F(t_m,\lambda_m) -\mathcal F(t_*,\lambda_*),(v,w)    \rangle
	+ \langle \mathcal F(t_*,\lambda_*),(v,w)    \rangle\\
	&\geq \gamma \Vert (v,w)\Vert_{\mathcal V\times \mathcal
		V}^2 -L \kappa_d \Vert (v,w)\Vert_{\mathcal V\times
		\mathcal V}.
	\end{split}
	\end{equation*}
	Since $\gamma R - L \kappa_d = \gamma  \delta - \kappa_d (\gamma +L)\geq 0$, we can conclude 
	$\langle h_d(\vec x),\vec x\rangle = R(\gamma R -
	L\kappa_d)\geq 0$.  Lemma~\ref{thm:brouwer}
	now establishes that $h_d(\vec x_*)=0$ for some $\vec x_*$ 
	with $\Vert \vec x \Vert_{2d} = \Vert (v_*,w_*)\Vert_{\mathcal V\times \mathcal V} \leq  R $, which is equivalent to that 
	$(t_d,\lambda_d) := (t_m + v_*,\lambda_m + w_*) $ solves the projected problem $\mathcal{F}_d =0$. The uniqueness follows from Theorem \ref{UandE} applied to $\mathcal{F}_d$. 
\end{proof}

We will next show the power of the bivariational principle as
far as the ECC method is concerned. The standard variational
formulation of CC theory introduces a Lagrangian. Error estimates for the CC
energy then requires that the dual problem has a solution. (See
\cite{Rohwedder2013b} where this non-trivial step has been done
by means of the Lax--Milgram theorem.)
However, the ECC method is based on the bivariational
principle and the energy itself is stationary in this formulation, i.e.,  
the solution $(t_*,\lambda_*)$ is a
critical point of the bivariational energy. When
$(t_d,\lambda_d)$ is close to the exact solution, we are guaranteed
a quadratic error estimate for free. As our last order of business we will discuss this further.

Under the assumption that $ H$
supports a ground state with ground-state energy $E_*$, the
Rayleigh-Ritz variational principle states that
\begin{equation*}
E_* \leq \mathcal E_\text{var}(\psi) :=\frac{\langle\psi, H\psi \rangle}{\langle\psi,\psi \rangle} 
\label{VarPrin}
\end{equation*}
for any $\psi\in\mathcal H$. Minimizing $\mathcal E_\text{var}$ over
trial wavefunctions (say, considering
$\mathcal H_\text{appr} \subset \mathcal H$) yields an approximate
energy $E_\text{appr}$ that also provides an upper bound to $E_*$,
i.e., $E_\text{appr}\geq E_*$. Furthermore, since
$D_\psi \mathcal E_\text{var}(\psi_*)=0$, we obtain a second-order
error estimate of the energy (see for instance Eq. (1.4) in
\cite{Rohwedder2013b} and the reference given in connection for more
refined estimates)
\[
0\leq E_\text{appr} - E_* \leq C \Vert \psi_\text{appr} - \psi_* \Vert_{\mathcal H}^2 \leq C' d(\mathcal H_\text{appr}, \psi_*)^2.
\]
In similar a fashion, the critical point condition $D \mathcal E_\text{bivar}(\psi_*,\psi_*') =0$ of the bivariational quotient will give us a second-order error estimate 
of the ECC energy.

As far as truncations of the double wavefunction space $\mathcal{M}
\subset \mathcal H \times \mathcal H$ is concerned (see
Eq.~\eqref{eq:M-def}), where the bivariational pair $(\psi,\tilde{\psi})$ is an element, we will use 
\[
\mathcal M_d:= \{ (\psi,\tilde \psi) : \psi = e^T  \phi_0, \tilde \psi = e^{-T^\dagger}  e^\Lambda \phi_0, \quad t,\lambda\in \mathcal V_d \} .
\]
Since $\mathcal M_d$ is closed (we assume that $\mathcal V_d$ is closed, see the next lemma), we define the distance
\begin{align*}
d( \mathcal M_d, (\psi_*,\tilde \psi_*)) := \min_{(\psi,\tilde{\psi}) \in \mathcal M_d} \Vert (\psi,\tilde{\psi}) - 
(\psi_*,\tilde \psi_*) \Vert_{\mathcal H   \times \mathcal H},
\end{align*}
where $\Vert (\cdot,\cdot \cdot) \Vert_{\mathcal H \times \mathcal H}^2:= 
\Vert \cdot \Vert_{\mathcal H}^2 + \Vert \cdot\cdot \Vert_{\mathcal H}^2$. 
\begin{lemma}
	Assume that $\mathcal V_d$ is closed. Then $\mathcal M_d$ is closed. Moreover, 
	it holds 
	\begin{equation}
	d(\mathcal V_d,t_* )^2 + d(\mathcal V_d,\lambda_* )^2 \leq C \, d( \mathcal M_d, (\psi_*,\tilde \psi_*))^2
	\label{HtimesHd}
	\end{equation}
	for some constant $C$.
	\label{LemmaHtimesHd}
\end{lemma}
\begin{proof}
	By Lemma \ref{lemma:ECCpara}, the map $\Phi:(t,\lambda) \mapsto (e^T\phi_0,e^{-T^\dagger} e^\Lambda\phi_0)$ and its inverse are smooth and $\mathcal M_d = \Phi(\mathcal V_d\times\mathcal V_d)$ is closed since $\mathcal V_d$ is.

	For \eqref{HtimesHd}, we first note that
	\begin{equation*}
	\begin{split}
	d(  \mathcal M_d, (\psi_*,\tilde \psi_*))^2 &= 
	\min_{t,\lambda \in \mathcal V_d } \big( \Vert e^T \phi_0 -e^{T_*} \phi_0 \Vert_{\mathcal H}^2 \\
	& \quad +   \Vert e^{-T^\dagger} e^\Lambda  \phi_0 -e^{-T_*^\dagger} e^{\Lambda_*} \phi_0   \Vert_{\mathcal H}^2 \big).
	\end{split}
	\end{equation*}
	This gives (where we let $C$ be a constant that is redefined and reused at leisure)
	\begin{equation*}
	\begin{split}
	d(\mathcal V_d, \lambda_*)^2 &\leq C \min_{\lambda \in \mathcal V_d} \Vert e^\Lambda \phi_0 - e^{\Lambda_*} \phi_0 \Vert_{\mathcal H}^2 \\
	& \leq C \big( \min_{t,\lambda\in\mathcal V_d}   \Vert e^{T_*^\dagger} \Vert_{\mathcal B (\mathcal H)}^2  \big( \Vert   
	e^{-T^\dagger}  e^{\Lambda} \phi_0 - e^{-T_*^\dagger}  e^{\Lambda_*} \phi_0 \Vert_{\mathcal H}^2 \\
	& \quad +  \Vert e^{-T^\dagger} - e^{-T_*^\dagger} \Vert_{\mathcal B (\mathcal H)} \Vert e^\Lambda \Vert_{\mathcal B (\mathcal H)}^2  \big)   \big) \\
	& \leq C \big( \min_{t,\lambda\in\mathcal V_d}  \Vert   
	e^{-T^\dagger}  e^{\Lambda} \phi_0 - e^{-T_*^\dagger}  e^{\Lambda_*} \phi_0 \Vert_{\mathcal H}^2   
	+ \min_{t\in\mathcal V_d}\Vert e^T\phi_0 -e^{T_*}\phi_0 \Vert_{\mathcal H}^2    \big) \\
	& \leq C\, d( (\mathcal H \times \mathcal H)_d, (\psi_\perp,\tilde \psi_\perp))^2.
	\end{split}
	\end{equation*}
	The desired inequality then follows from,
	\[
	d(\mathcal V_d,t_*) \leq D \min_{t\in\mathcal V_d} || e^T\phi_0 - e^{T_*}\phi_0 ||_\mathcal{H}   \leq D\, 
	d( \mathcal M_d, (\psi_*,\tilde \psi_*)).
	\]
\end{proof}
\begin{theorem}
	\label{thm:Eest}
	Let $\delta>0$ be such that $\mathcal F$ is strongly monotone (with constant $\gamma$) and Lipschitz continuous (with constant $L$) for $(t,\lambda)\in B_\delta(t_*,\delta_*)$ and 
	assume that $\mathcal V_d$ is sufficiently good an approximation of $\mathcal V$. 
	If  $(t_d,\lambda_d)\in\mathcal V_d \times \mathcal V_d$ is the solution of $\mathcal F_d =0$ and $(t_*,\lambda_*)\in\mathcal V \times \mathcal V$ is the (exact) solution of $\mathcal F=0$, then:
	\begin{itemize}
		\item[(i)]
		With $E_d:= \mathcal E(t_d,\lambda_d)$ there exist constants $d_1,d_2$ such that
		\begin{equation}
		|E_d - E_*| \leq d_1 \Vert t_d -t_*\Vert_{\mathcal V}^2 + d_2 \Vert t_d -t_*\Vert_{\mathcal V}\Vert \lambda_d -\lambda_*\Vert_{\mathcal V}
		\label{Eest1}
		\end{equation}
		and with $C_*$ as before there holds
		\begin{equation}
		\begin{split}
		|E_d - E_*|&\leq (C_*  + \mathcal O(\Vert t_*\Vert_{\mathcal V}) + \mathcal O(\Vert \lambda_*\Vert_{\mathcal V}))   \frac{L^2}{2\gamma^2} \big( d(\mathcal V_d,t_*)^2 + d(\mathcal V_d,\lambda_*)^2\big)\\
		& \quad + \mathcal O \big( \max(d(\mathcal V_d,t_*) , d(\mathcal V_d,\lambda_*))^3\big).
		\end{split}
		\label{Eest2}
		\end{equation}
		\item[(ii)] Letting $\psi_* = e^{T_*}\phi_0$, $\psi_d = e^{T_d}\phi_0$, $\tilde \psi_* = e^{-T_*^\dagger} e^{\Lambda_*} \phi_0$ and 
		$\tilde \psi_d = e^{-T_d^\dagger} e^{\Lambda_d} \phi_0$, there exist $\tilde d_1, \tilde d_2$ such that
		\begin{equation}
		|E_d - E_*| \leq \tilde d_1 \Vert \psi_d -\psi_*\Vert_{\mathcal H}^2 + \tilde d_2 \Vert \psi_d -\psi_*\Vert_{\mathcal H}\Vert \tilde \psi_d -\tilde \psi_*\Vert_{\mathcal H}.
		\label{Eest3}
		\end{equation}
		Furthermore, there exists a constant $\tilde C$ such that
		\begin{equation}
		|E_d-E_*| \leq \tilde C \,  d( \mathcal M_d, (\psi_*,\tilde{\psi}_*))^2 + \mathcal O \big( d(\mathcal M_d,(\psi_*,\tilde \psi_*) )^3\big).
		\label{Eest5}
		\end{equation}
	\end{itemize}
\end{theorem}
\begin{proof}
	(i) Taylor expanding $\mathcal E(t,\lambda)$ at $(t_*,\lambda_*)$ and using the notation $g_d:= t_d-t_*$ and $k_d := \lambda_d-\lambda_*$, we obtain 
	(by Taylor's theorem)
	\begin{equation*}
	\begin{split}
	E_d - E_*  &=  \frac{1}{2}   D^2  \mathcal E(t_*,\lambda_*)((g_d,k_d)^2)\\
	&\quad + \frac 1 2\int_{0}^1 (1-r)^2 D^3 \mathcal E ((t_*,\lambda_*) + r(g_d,k_d)) ((g_d,k_d)^3)  dr .
	\end{split}
	\end{equation*}
	From this it is clear that 
	\begin{equation}
	2|E_d - E_*|\leq |D^2  \mathcal E(t_*,\lambda_*)((g_d,k_d)^2)| + \mathcal O\big( \max(d(\mathcal V_d,t_*) , d(\mathcal V_d,\lambda_*))^3\big).
	\label{TaylorsThm}
	\end{equation}
	By straightforward differentiation with respect to the
	amplitudes $t_\mu$ and $\lambda_\mu$,
	\begin{align*}
	(D^2\mathcal E(t,\lambda))_{\mu,\nu} =
	\begin{bmatrix}
	\langle  \phi_0, e^{ \Lambda^\dagger} [[ H_t, X_\mu], X_\nu] \phi_0\rangle & 
	\langle \phi_\nu, e^{ \Lambda^\dagger} [ H_t, X_\mu] \phi_0\rangle \\
	\langle \phi_\mu, e^{ \Lambda^\dagger} [ H_t, X_\nu] \phi_0\rangle & 
	\langle X_\mu X_\nu \phi_0, e^{ \Lambda^\dagger}  H_t\phi_0\rangle
	\end{bmatrix}.
	\end{align*}
	We next note that
	\begin{align*}
	\frac 1 2 & D^2\mathcal E(t_*,\lambda_*) ((g_d,k_d )^2)  \\
	&=  \frac 1 2 \big(  \langle \phi_0, e^{ \Lambda_*^\dagger} [ [ H_{t_*}, G_d],  G_d ]  \phi_0 \rangle 
	+ 2 \langle  K_d\phi_0, e^{ \Lambda_*^\dagger} [ H_{t_*},  G_d] \phi_0\rangle  + \langle  K_d^2 \phi_0, e^{ \Lambda_*^\dagger}  H_{t_*}\phi_0\rangle\big)   \\
	& =\frac 1 2 \big( \langle e^{ \Lambda_*} \phi_0, \big(  H_{t_*}  G_d^2 - 2  G_d  H_{t_*}  G_d +  G_d^2  H_{t_*}   \big) \phi_0\rangle    
	+ 2 \langle e^{ \Lambda_*}  K_d\phi_0, [ H_{t_*} ,  G_d] \phi_0\rangle  \\
	& \qquad + \langle e^{ \Lambda_*} K_d^2 \phi_0,  H_{t_*} \phi_0 \rangle \big).
	\end{align*}
	Using Lemma \ref{SimSEs}, specifically $ H_{t_*} \phi_0 = E_* \phi_0$ and $ H_{t_*}^\dagger e^{ \Lambda_*} \phi_0 = E_* e^{ \Lambda_*} \phi_0$, the following equality holds
	\begin{align*}
	\frac 1 2 & D^2\mathcal E(t_*,\lambda_*) ((g_d,k_d )^2)\\
	&= \frac 1 2 \big(  2 \langle e^{ \Lambda_*}\phi_0,  G_d(E_* - H_{t_*} )  G_d \phi_0 \rangle + 2 \langle e^{ \Lambda_*}  K_d\phi_0,( H_{t_*} -E_*)  G_d\phi_0\rangle   \big) .
	\end{align*}
	Furthermore, since $e^{ \Lambda_*}$ and $ K_d$ commute, we obtain 
	\begin{equation}
	\begin{split}
	\frac 1 2 & \vert D^2\mathcal E(t_*,\lambda_*) ((g_d,k_d )^2)\vert \\
	&=  \vert \langle  G_d^\dagger e^{ \Lambda_*} \phi_0,(E_* -  H_{t_*} )  G_d \phi_0 \rangle + \langle e^{ \Lambda_*}  K_d\phi_0,( H_{t_*} -E_*)  G_d\phi_0\rangle \vert  \\
	& = \vert\langle e^{-T_*^\dagger} \big(G_d^\dagger (e^{\Lambda_*} - I)  - e^{ \Lambda_*}  K_d   \big)\phi_0, (E_* - H) e^{T_*} G_d \phi_0 \rangle \vert \\
	& \leq C_* \Vert  e^{-T_*^\dagger} \big( G_d^\dagger(e^{\Lambda_*} -I) - e^{\Lambda_*} K_d   \big) \phi_0 \Vert_{\mathcal H} 
	\Vert e^{T_*} G_d \phi_0 \Vert_{\mathcal H} \\
	&\leq C_* \Vert  e^{-T_*^\dagger}\Vert_{\mathcal B(\mathcal H)} \big(  \Vert G_d^\dagger  \Vert_{\mathcal B(\mathcal H)} \Vert  e^{\Lambda_*} -I\Vert_{\mathcal B(\mathcal H)} \Vert \phi_0 \Vert_{\mathcal H} \\
	& \quad + \Vert  e^{\Lambda_*}\Vert_{\mathcal B(\mathcal H)} \Vert K_d \phi_0 \Vert_{\mathcal H}      \big)
	\Vert  e^{T_*}\Vert_{\mathcal B(\mathcal H)}\Vert G_d \phi_0 \Vert_{\mathcal H} \\
	& \leq C_* \Vert  e^{-T_*^\dagger}\Vert_{\mathcal B(\mathcal H)} \Vert  e^{T_*}\Vert_{\mathcal B(\mathcal H)}
	\big( c\Vert \phi_0 \Vert_{\mathcal H}  \Vert  e^{\Lambda_*} -I\Vert_{\mathcal B(\mathcal H)} \Vert t_d-t_* \Vert_\mathcal{V}^2 \\
	& \quad + \Vert  e^{\Lambda_*}\Vert_{\mathcal B(\mathcal H)} \Vert t_d-t_* \Vert_\mathcal{V}\Vert \lambda_d-\lambda_* \Vert_\mathcal{V}\big)\\
	& =: D_1 \Vert t_d-t_* \Vert_\mathcal{V}^2 +  D_2 \Vert t_d-t_* \Vert_\mathcal{V}\Vert \lambda_d-\lambda_* \Vert_\mathcal{V},
	\end{split}
	\label{lemmaEq}
	\end{equation}
	where we in the last step defined the constants $D_1 := D_1 (t_*,\lambda_*,\phi_0)$ and $D_2 := D_2(t_*,\lambda_*)$. Thus, by \eqref{TaylorsThm} we can choose 
	$d_1$ and $d_2$, under the assumption that $\max(d(\mathcal V_d,t_*) , d(\mathcal V_d,\lambda_*))$ is sufficiently small, such that \eqref{Eest1} holds.
	
	To obtain \eqref{Eest2}, we see that \eqref{lemmaEq} gives
	\begin{align*}
	\frac 1 2 & \vert D^2\mathcal E(t_*,\lambda_*) ((g_d,k_d )^2) \vert \\
	& \leq  C_* \Vert  e^{-T_*^\dagger}\Vert_{\mathcal B(\mathcal H)} \Vert  e^{T_*}\Vert_{\mathcal B(\mathcal H)}
	\big( c\Vert \phi_0 \Vert_{\mathcal H}  \Vert  e^{\Lambda_*} -I\Vert_{\mathcal B(\mathcal H)} + \frac 1 2 \Vert e^{\Lambda_*} \Vert_{\mathcal B (\mathcal H)}   \big)\\ & \quad \times \big( \Vert t_d-t_* \Vert_\mathcal{V}^2 + \Vert \lambda_d-\lambda_* \Vert_\mathcal{V}^2  \big) \\
	&\leq  (C_* + \mathcal O(\Vert t_* \Vert_{\mathcal V})+ \mathcal O(\Vert \lambda_* \Vert_{\mathcal V}) )\frac{L^2}{2 \gamma^2}\big( d(\mathcal V_d,t_*)^2 + d(\mathcal V_d,\lambda_*)^2\big),
	\end{align*}
	where we used \eqref{Est2}.
	
	(ii) Next, using Theorem \ref{Thm1RS} (equation \eqref{ExpMap1}), \eqref{lemmaEq} gives
	\begin{equation}
	\begin{split}
	\frac 1 2  \vert D^2\mathcal E(t_*,\lambda_*) ((g_d,k_d )^2)\vert 
	\leq \tilde D_1 \Vert \psi_d-\psi_* \Vert_\mathcal{H}^2 + \tilde  D_2 \Vert \psi_d-\psi_* \Vert_\mathcal{H}\Vert (e^{\Lambda_d}- e^{\Lambda_*})\phi_0 \Vert_\mathcal{H}.
	\end{split}
	\label{tildeEq1}
	\end{equation}
	Furthermore, we use
	\begin{equation*}
	\begin{split}
	e^{\Lambda_d} - e^{\Lambda_*} &= e^{T_*^\dagger} e^{-T_*^\dagger} (e^{\Lambda_d} - e^{\Lambda_*} ) \\
	&=e^{T_*^\dagger} \big(  e^{-T_d^\dagger}e^{\Lambda_d} - e^{-T_*^\dagger} e^{\Lambda_*} - (e^{-T_d^\dagger} - e^{-T_*^\dagger}) e^{\Lambda_d}   \big),
	\end{split}
	\end{equation*}
	and we obtain
	\begin{equation}
	\begin{split}
	\Vert (e^{\Lambda_d}- e^{\Lambda_*})\phi_0 \Vert_\mathcal{H}&\leq  \Vert e^{T_*^\dagger} \Vert_{\mathcal B (\mathcal H)} 
	\big(\Vert  \tilde \psi_d  - \tilde \psi_*  \Vert_{\mathcal H}    + \Vert e^{\Lambda_d} \Vert_{\mathcal B (\mathcal H)} \Vert (e^{-T_d^\dagger} - e^{-T_*^\dagger}) \phi_0 	 \Vert_{\mathcal H} \big) \\
	&  \leq  \tilde D  \Vert \tilde \psi_d  - \tilde \psi_*  \Vert_{\mathcal H} + \tilde D' \Vert \psi_d - \psi_* \Vert_{\mathcal H} .
	\end{split}
	\label{tildeEq2}
	\end{equation}
	Inserting \eqref{tildeEq2} into \eqref{tildeEq1}, gives
	\[
	\frac 1 2  \vert D^2\mathcal E(t_*,\lambda_*) ((g_d,k_d )^2)\vert \leq \tilde D_1' \Vert \psi_d -\psi_*\Vert_{\mathcal H}^2 + \tilde D_2' \Vert \psi_d -\psi_*\Vert_{\mathcal H}\Vert \tilde \psi_d -\tilde \psi_*\Vert_{\mathcal H}.
	\]
	Repeating the argument made in (i) for \eqref{Eest1}, we can find constants $\tilde d_1, \tilde d_2$ such that \eqref{Eest3} holds. 
	
	To finish the proof, we use \eqref{HtimesHd} in Lemma \ref{LemmaHtimesHd} that together with the proof of (i) give \eqref{Eest5}.
\end{proof}

\section{Conclusions}
\label{sec:conclusions}

In this article we have put the formalism of Arponen's ECC method on
firm mathematical ground. This has been achieved by generalizing the
continuous (infinite dimensional) formulation of standard CC theory in
Refs.~\cite{Rohwedder2013,Rohwedder2013b} to the ECC
formalism. The bivariational principle plays an important
role in our analysis. With the bivariational energy $\mathcal E (t,\lambda)$ (and its
derivatives) as the main object of study, we have derived existence
and uniqueness results for the extended CC equation $\mathcal F =0$
(the flipped gradient) and its discretizations $\mathcal F_d =0$. 
The key
aspect of the analysis is the establishment of locally strong
monotonicity of $\mathcal F$ at the exact solution
$(t_*,\lambda_*)$. This has been achieved by either assuming that the
reference $\phi_0$ is sufficiently good an approximation of the exact
solution $\psi_*$, or by considering certain splittings of the
Hamiltonian $H$.

 We have formulated and proved
quadratic error estimates in terms of the quality of the truncated
amplitude space $\mathcal V_d$. The energy error has been bound in
terms of $d(\mathcal V_d,t_*)$ and $d(\mathcal V_d,\lambda_*)$, or
equivalently $d( \mathcal M_d,(\psi_*,\tilde{\psi}_*))$, where
$(\psi_*,\tilde{\psi}_*)$ is the exact wavefunction pair and
$\mathcal M_d$ the truncation of $\mathcal H \times \mathcal H$. 

It is interesting to note, as ECC is variational by construction,
i.e., the solution $(t_*,\lambda_*)$ is a critical point of the smooth
map $\mathcal{E}$, that the error estimate is obtained basically for
free. Indeed, the CC Lagrangian $\mathcal L$ can be thought of as a
linearized formulation of ECC where the second set of amplitudes
$\{\lambda_\mu\}$ are the Lagrange multipliers $\{z_\mu\}$. The dual
problem of CC is, as it were, already built into the ECC theory. This
again illustrates the benefit of applying the bivariational point of
view. 

Here, ECC has been formulated in a set of cluster amplitude coordinates that are
not usually employed. A next step in the study of the ECC method would
be to repeat the analysis of the monotonicity of $\mathcal{F}$ and to
obtain error estimates using the so-called canonical cluster
amplitudes, cf.~Remark~\ref{canonical-remark}.

Even if ECC is currently not a practical tool in computational
chemistry due to its complexity, our analysis demonstrates an
important fact: The bivariational principle can be utilized to devise
computational schemes that are not obtainable from the standard
Rayleigh--Ritz principle, but still have a quadratic error
estimate. Such schemes include both the traditional CC method
and the ECC method. Indeed, not being variational in the
Rayleigh--Ritz sense has been the single most important critique of the
coupled-cluster method, precisely due to the lack of a quadratic error
estimate. Moreover, we believe that the approach taken in this article, by
showing the monotonicity of the flipped gradient $\mathcal{F}$, is an
approach that may allow existence and uniqueness results in
much more general settings.


\begin{thebibliography}{10}
	\bibitem{Arponen1983}
	{\sc J.~Arponen}, {\em {Variational principles and linked-cluster exp S
			expansions for static and dynamic many-body problems}}, Annals of Physics,
	151 (1983), pp.~311--382.
	
	\bibitem{Bishop1991}
	{\sc R.~Bishop}, {\em An overview of coupled cluster theory and its
		applications in physics}, Theor. Chim. Acta, 80 (1991), pp.~95--148.
	
	\bibitem{Cizek1966}
	{\sc J.~{\v C}{\'\i}{\v z}ek}, {\em {On the Correlation Problem in Atomic and
			Molecular Systems. Calculation of Wavefunction Components in Ursell-Type
			Expansion Using Quantum-Field Theoretical Methods}}, J. Chem. Phys., 45
	(1966), pp.~4256--4266.
	
	\bibitem{Cizek1991}
	{\sc J.~{\v C}{\'\i}{\v z}ek}, {\em Origins of the coupled cluster technique
		for atoms and molecules}, Theor. Chim. Acta, 80 (1991), pp.~91--94.
	
	\bibitem{Coester1958}
	{\sc F.~Coester}, {\em Bound states of a many-particle system}, Nucl.~Phys., 7
	(1958), pp.~421--424.
	
	\bibitem{Coester1960}
	{\sc F.~Coester and H.~K\"ummel}, {\em Short-range correlations in nuclear wave
		functions}, Nucl.~Phys., 17 (1960), pp.~477--485.
	
	\bibitem{Dean2004a}
	{\sc D.~J. Dean and M.~Hjorth-Jensen}, {\em Coupled-cluster approach to nuclear
		physics}, Phys. Rev. C, 69 (2004), p.~054320.
	
	\bibitem{Emmrich2004}
	{\sc E.~Emmrich}, {\em {Gew{\"o}hnliche und Operator-Differentialgleichungen}},
	Vieweg, Wiesbaden, Germany, 2004.
	
	\bibitem{Helgaker1988}
	{\sc T.~Helgaker and P.~J{\o}rgensen}, {\em {Analytical Calculation of
			Geometrical Derivatives in Molecular Electronic Structure Theory}}, Adv.
	Quant. Chem., 19 (1988), pp.~183--245.
	
	\bibitem{Helgaker1989}
	{\sc T.~Helgaker and P.~J{\o}rgensen}, {\em Configuration-interaction energy
		derivatives in a fully variational formulation}, Theor. Chim. Acta, 75
	(1989), pp.~111--127.
	
	\bibitem{Kummel1991}
	{\sc H.~K{\"u}mmel}, {\em Origins of the coupled cluster method}, Theor. Chim.
	Acta, 80 (1991), pp.~81--89.
	
	\bibitem{Kvaal2012}
	{\sc S.~Kvaal}, {\em Ab initio quantum dynamics using coupled-cluster}, J.
	Chem. Phys., 136 (2012), p.~194109.
	
	\bibitem{Kvaal2013}
	{\sc S.~Kvaal}, {\em Variational formulations of the coupled-cluster method in
		quantum chemistry}, Mol. Phys., 111 (2013), pp.~1100--1108.
	
	\bibitem{Paldus2005}
	{\sc J.~Paldus}, {\em The beginnings of coupled-cluster theory: an eyewitness
		account}, in Theory and Applications of Computational Chemistry: The First
	Forty Years, C.~Dykstra, G.~Frenking, K.~Kim, and G.~Scuseria, eds.,
	Elsevier, 2005, ch.~7, p.~115.
	
	\bibitem{Paldus1972}
	{\sc J.~Paldus, J.~{\v C}{\'\i}{\v z}ek, and I.~Shavitt}, {\em {Correlation
			Problems in Atomic and Molecular Systems. IV. Extended Coupled-Pair
			Many-Electron Theory and Its Application to the BH3 Moleciule}}, Phys. Rev.
	A, 5 (1972), pp.~50--67.
	
	\bibitem{Rohwedder2013}
	{\sc T.~Rohwedder}, {\em {The continuous coupled cluster formulation for the
			electronic Schr{\"o}dinger equation}}, ESAIM: Math. Mod. Num. Anal., 47
	(2013), pp.~421--447.
	
	\bibitem{Rohwedder2013b}
	{\sc T.~Rohwedder and R.~Schneider}, {\em Error estimates for the coupled
		cluster method}, ESAIM: Math. Mod. Num. Anal., 47 (2013), pp.~1553--1582.
	
	\bibitem{Schneider2009}
	{\sc R.~Schneider}, {\em {Analysis of the projected Coupled Cluster Method in
			Electronic Structure Calculation}}, Numer. Math., 113 (2009), pp.~433--471.
	
	\bibitem{Sinanoglu1962}
	{\sc O.~Sinano\u{g}lu}, {\em {Many-Electron Theory of Atoms and Molecules. I.
			Shells, Electron Pairs vs. Many-Electron Correlations}}, J. Chem. Phys., 36
	(1962), pp.~706--717.
	
	\bibitem{Yserentant2010}
	{\sc H.~Yserentant}, {\em Regularity and approximability of electronic
		wavefunctions}, Lecture Notes In Mathematics, Springer, New York, Heidelberg,
	Berlin, 2010.
	
	\bibitem{Zarantonello1960}
	{\sc E.~Zarantonello}, {\em {Solving functional equations by contractive
			averaging}}, Tech. Report 160, U.S. Army Math. Res. Centre, Madison, WI.,
	1960.
	
	\bibitem{Zeidler1990}
	{\sc E.~Zeidler}, {\em {Nonlinear Functional Analysis and its Application
			II/B}}, Springer, New York, Heidelberg, Berlin, 1990.
\end{thebibliography}

\end{document}